\documentclass[a4paper,11pt,reqno]{amsart}
\usepackage{amsmath,amsthm,amssymb,cite,tikz,pgf,bbm,framed,enumerate,verbatim,fancyhdr}

\usepackage{mathptmx}
\usepackage[colorlinks=true]{hyperref} 
\usepackage{enumitem}

\usetikzlibrary{matrix,arrows,cd}
\usepackage[matrix,arrow,curve,frame]{xy}    

\xymatrixcolsep{1.9pc}                          
\xymatrixrowsep{1.9pc}
\newdir{ >}{{}*!/-5pt/\dir{>}}                  

 \calclayout

\newtheorem{theorem}{Theorem}[section]
\newtheorem{corollary}[theorem]{Corollary}
\newtheorem{lemma}[theorem]{Lemma}
\newtheorem{thm}{Theorem}
\newtheorem{proposition}[theorem]{Proposition}
\theoremstyle{definition}
\newtheorem{definition}[theorem]{Definition}
\newtheorem{remark}[theorem]{Remark}

\newtheorem{example}[theorem]{Example}

\newcommand{\bB}[1]{\ensuremath{\mathbb{#1}}}
\newcommand{\mC}[1]{{\ensuremath{\mathcal{#1}}}}

\newcommand{\eps}{\ensuremath{\varepsilon}}
\newcommand{\To}{\ensuremath{\Rightarrow}}

\newcommand{\bendL}[1]{\arrow[l,"{#1}" above,shift right, end
anchor=north east,bend right, start anchor=north west]}
\newcommand{\bendR}[1]{\arrow[l,"{#1}" below,shift left,end
anchor=south east,bend left, start anchor=south west]}

\DeclareMathOperator{\Hom}{Hom}
\DeclareMathOperator{\id}{id}

\DeclareMathOperator{\Ker}{Ker}
\DeclareMathOperator{\Coker}{Coker}
\DeclareMathOperator{\act}{act}
\DeclareMathOperator{\coact}{coact}
\DeclareMathOperator{\op}{op}
\DeclareMathOperator{\Ob}{Ob}
\DeclareMathOperator{\tens}{\otimes}

\DeclareMathOperator{\Ch}{\bf Ch}
\newcommand{\K}{\ensuremath{\textbf{K}}}
\newcommand{\cc}{\mathcal}
\DeclareMathOperator{\D}{\bf D}

\DeclareMathOperator{\Ab}{\bf Ab}

\DeclareMathOperator{\Ext}{Ext}
\DeclareMathOperator{\End}{End}
\let\Im\relax

\DeclareMathOperator{\Im}{Im}

\newenvironment{diagram*}{\begin{equation*}\begin{tikzcd}}{\end{tikzcd}\end{equation*}\break}
\tikzcdset{arrows={line width=rule_thickness}}

\begin{document}
\footskip30pt
\sloppy

\title{Recollements for derived categories of enriched functors and triangulated categories of motives}

\author{Grigory Garkusha}
\address{Department of Mathematics, Swansea University, Fabian Way, Swansea SA1 8EN, UK}
\curraddr{}
\email{g.garkusha@swansea.ac.uk}
\thanks{}

\author{Darren Jones}
\address{Department of Mathematics, Swansea University, Fabian Way, Swansea SA1 8EN, UK}
\curraddr{}
\email{darrenalexanderjones@gmail.com}
\thanks{The second author thanks the Swansea Science Doctoral Training Centre, 
and the Engineering and Physical Sciences Research Council for support.}

\subjclass[2010]{13D09, 14F42, 18D20}

\date{}

\keywords{Recollements of Abelian and triangulated categories, triangulated categories of motives}

\begin{abstract}
We investigate certain categorical aspects of Voevodsky's triangulated categories of motives. For this,
various recollements for Grothendieck categories of enriched functors and their derived categories 
are established. In order to extend these recollements further with respect to Serre's localization, the concept 
of the (strict) Voevodsky property for Serre localizing subcategories is introduced. This concept is inspired
by the celebrated Voevodsky theorem on homotopy invariant presheaves with transfers. As an application,
it is shown that Voevodsky's triangulated categories of motives fit into recollements of derived categories of 
associated Grothendieck categories of Nisnevich sheaves with specific transfers. 
\end{abstract}

\maketitle

\thispagestyle{empty} \pagestyle{plain}


\section{Introduction}

Triangulated categories of motives $\mathbf{DM}^{eff}_{\mC C}(k)$ over a (perfect) field $k$ constructed by Voevodsky~\cite{Voe2}
are of fundamental importance in motivic homotopy theory (here $\cc C$ is a reasonable category of correspondences
on smooth algebraic varieties $Sm/k$). By definition~\cite{Voe2}, it is the full triangulated subcategory of the derived category $\D\bigl(Shv(\mC C)\bigr)$
of complexes of Nisnevich $\cc C$-sheaves whose cohomology sheaves are homotopy invariant. By a theorem of Voevodsky~\cite{Voe2},
the inclusion $\mathbf{DM}^{eff}_{\mC C}(k)\to\D\bigl(Shv(\mC C)\bigr)$ has the left adjoint given by the Sulin complex functor $C_*$.
The category $\mathbf{DM}^{eff}_{\mC C}(k)$ is also equivalent to the quotient category of $\D\bigl(Shv(\mC C)\bigr)$
with respect to the localizing subcategory generated by complexes of the form $\cc C(-,X\times\mathbb A^1)_{nis}\to\cc C(-,X)_{nis}$
(see~\cite{SV1,Voe2} for details).

Inspired by Voevodsky's constructions~\cite{Voe2}, we investigate certain categorical aspects of $\mathbf{DM}^{eff}_{\mC C}(k)$.
For this purpose, we work in the framework of Grothendieck categories of enriched functors in the sense of~\cite{AlGG} 
and their derived categories\cite{GGME}. We shall convert some fundamental Voevodsky's theorems into the language of enriched
category theory and arrive at some categorical concepts and results which are of independent interest. Afterwards, we shall
translate these concepts and results back into the motivic language.

First, recall from~\cite{AlGG} that the category of enriched functors $[\mC C,\mC V]$, where \mC V is a closed symmetric monoidal Grothendieck
category and \mC C is a small \mC V-category, is a Grothendieck $\mC
V$-category with a set of generators $\{\mC V(c,-) \oslash g_i\mid c\in\Ob\mC C,i\in I\}$, where $\{g_i\}_I$ is a set of generators
of $\mC V$\!. As in~\cite{AlGG}, we shall refer to the category $[\mC C,\mC V]$ as the \emph{Grothendieck category
of enriched functors}. Our first result is as follows (see Theorem~\ref{Egg}).

\begin{thm}\label{introEgg}
Suppose $\mC A$ is a full \mC V-subcategory of a small \mC V-category \mC C.
Define a localizing subcategory 
$\mC S_\mC A:=\{Y\in[\mC C,\mC V]\mid Y(a)=0\textrm{ for all $a\in\mC A$}\}\subset [\mC C,\mC V].$ There is a recollement \begin{diagram*}[column
sep=huge] \mC S_\mC A\arrow[r,"i"] &{[\mC C, \mC V]}
   \bendR{i_R}
   \bendL{i_L}
   \arrow[r,"r"]			
& {[\mC A, \mC V]}\bendR{r_R} 
   \bendL{r_L} 
\end{diagram*}
with functors $i,r$ being the canonical inclusion and restriction functors
respectively. The functors $r_L,r_R$ are the enriched left and right Kan extensions respectively,
$i_R$ is the torsion functor associated with the localizing subcategory $\mC S_\mC A$.
\end{thm}

This previous theorem can be used to recover an important example
of recollements~\cite[Chapter~IV]{Stein} of modules induced by idempotents (see Example~\ref{EGG}). Using this
theorem, we can extend~\cite[Theorem 5.3]{AlGG} into a recollement (see~Theorem~\ref{Egg2}).

\begin{thm}\label{introEgg2}
There is a recollement
\begin{diagram*}[column
sep=huge] \mC S_\mC A\arrow[r,"i"] &{[\mC C, \mC V]}
   \bendR{i_R}
   \bendL{i_L}
   \arrow[r,"\ell"]			
& {[\mC C,\mC V]/\mC S_\mC A}\bendR{\ell_R} 
   \bendL{\ell_L} 
\end{diagram*}
with functors $i_L,i,i_R$ being from the preceding theorem. The functor $\ell$ is the localization functor, 
$\ell_R$ is the inclusion and $\ell_L:=r_L\circ {(\ell\circ r_L)}^{-1}$.
\end{thm}

Next, we shall pass to derived categories. In what follows in the introduction, we shall always assume that
the derived category $\D(\cc V)$ is compactly generated with some reasonable assumptions on
compact generators (more precisely, we assume Theorem~\ref{cond} to satisfy). 
We prove that there is a recollement as follows (see Theorem~\ref{TA}).
 
\begin{thm}\label{intro3}
There exists
a recollement of triangulated categories 
\begin{diagram*}[column sep=huge]
\phantom{X}\mC E_\mC A\arrow[r,"\iota"] &{\D[\mC C,\mC V]}
   \arrow[l, shift left, bend left, "\iota_R"] 
   \arrow[l, shift right, bend right,swap ,"\iota_L"]
   \arrow[r,"\rho"]			
& {\D[\mC A,\mC V]}\arrow[l, shift left, bend left, "\rho_R"] 
   \arrow[l, shift right, bend right,swap,"\rho_L"] 
\end{diagram*}
where $\mC A \subset \mC C$, $\mC E_\mC A := \bigl\{Y\in \mC D[\mC C, \mC V] \mid H_n\bigl(Y(a)\bigr) = 0
\text{ for all } a\in \mC A\text{ and }n \in \bB Z\bigr\}$, $\iota$ is the inclusion, $\rho$ is the
restriction.
\end{thm}

Subsequently, we consider a Serre localizing subcategory $\mC Q\subset [\mC C,\mC V].$ It is not 
readily apparent when recollements as above are compatible with \mC Q-localization. 
Precisely, after constructing a recollement of $\D[\mC C,\mC V]$ in the preceding theorem, it is important for applications
to extend it further
to a recollement of $\D\bigl([\mC C,\mC V]/\mC Q\bigr)$.
Therefore to make the desired extension of the recollement to $\D\bigl([\mC C,\mC V]/\mC Q\bigr)$ possible,
we need to find the right conditions on the localizing subcategories $\mC S_{\mC A}$ and $\mC Q$
of $[\mC C,\mC V]$. These conditions originate in the fundamental Voevodsky theorem~\cite{Voe1}, which says that
the Nisnevich sheaf $F_{nis}$ associated with a homotopy invariant presheaf with transfers $F$ is 
a homotopy invariant sheaf with transfers
and that it is strictly homotopy invariant whenever the base field is perfect. We translate this theorem into the language
of Serre and Bousfield localization theory in Grothendieck categories and their derived categories. Namely, we formulate
the desired conditions in Definitions~\ref{vprop} and~\ref{vsprop} that $\mC S_\mC A$ and $\cc Q$ must satisfy,
that allow us to extend Theorem~\ref{intro3}. We shall call these conditions the ``Voevodsky properties" owing
to the influence of Voevodsky's fundamental theorem~\cite{Voe1} on homotopy invariant 
presheaves with transfers. We construct the following recollement (see Theorem~\ref{bigthm}). 

\begin{thm}\label{intro4}
Suppose $\mC A\subset\mC C$, and $\mC Q\subset [\mC C,\mC V]$ 
is a localizing subcategory such that: 
\begin{itemize}
  \item $\D \big([\mC C, \mC V]/\mC Q \big)$ is
  compactly generated and the functor induced by the exact $\mC Q$-localization functor respects compact objects;
  \item the localizing subcategory $\mC S_\mC A$ 
  satisfies the strict Voevodsky property with respect to $\mC Q$ (see  Definition~\ref{vsprop}).
\end{itemize}
Then there exists a recollement of triangulated categories \begin{diagram*}[column sep=huge]
\mC E_\mC A^\mC Q\arrow[r,"\iota^\mC Q"] &{\D\bigl([\mC C,\mC V]/\mC Q\bigr)}
   \bendR{\iota_R^\mC Q} 
   \bendL{\iota_L^\mC Q}
   \arrow[r,"\lambda^\mC Q"]			
& {\D\bigl([\mC C,\mC V]/\mC J_\mC A\bigr),}
   \bendR{\lambda_R^\mC Q} 
   \bendL{\lambda_L^\mC Q}
\end{diagram*}
where $\mC E_\mC A^{\mC Q}$ is the full subcategory containing chain complexes with homology belonging to the
\mC Q-localization of $\mC S_\mC A,$ and $\mC J_\mC A$ is the smallest localizing category containing $\mC Q$ and 
$\mC S_\mC A$ in $[\mC C,\mC V]$.
\end{thm}

The next step is to convert the previous theorem back into the motivic language,
and we are now in a position to discuss its main application.
In practice the derived category for the Grothendieck category $Shv(\mC C)$ of Nisnevich sheaves with reasonable 
transfers $\mC C$ on smooth algebraic varieties
$Sm/k$ plays the role of the middle category. In this language 
the left category will be nothing but Voevodsky's~\cite{Voe2} triangulated category of 
motives $\mathbf{DM}^{eff}_{\mC C}(k)$. Also, the Voevodsky theorem~\cite[3.2.6]{Voe2} computes
the functor $\iota_L^{\mC Q}$ as the Suslin complex $C_*$. We refer the reader to Section~\ref{voev} for 
details on how Theorem~\ref{intro4} encodes the above information about $\mathbf{DM}^{eff}_{\mC C}(k)$.

Therefore Theorem~\ref{intro4}
is a kind of categorical framework for the triangulated category of motives $\mathbf{DM}^{eff}_{\mC C}(k)$.
Precisely, the following theorem is true (see Theorem~\ref{shvthm}).

\begin{thm}\label{intro7}
Suppose $\mC C$ is a strict  $V$-category of correspondences in the sense of~\cite{Gar17}.
There exists a recollement of triangulated categories 
\begin{diagram*}[column sep=huge]
\mathbf{DM}^{eff}_{\mC C}(k)\arrow[r,"\iota^\mC Q"] &{\D\bigl(Shv(\mC C)\bigr)}
   \bendR{\iota_R^\mC Q} 
   \bendL{C_*}
   \arrow[r,"\lambda^\mC Q"]			
& {\D\bigl(Shv(\mC C)/\mC S_{\mathbb A^1}^{\mC Q}\bigr),}
   \bendR{\lambda_R^\mC Q} 
   \bendL{\lambda_L^\mC Q}
\end{diagram*}
where $\mC S_{\mathbb A^1}^{\mC Q}$ is some localizing Serre subcategory of $Shv(\mC C)$,
$C_*$ is the Suslin complex functor,
the functor $\lambda^{\mC Q}$ is induced by the $\mC S_{\mathbb A^1}^{\mC Q}$-localization functor 
$Shv(\mC C)\to Shv(\mC C)/\mC S_{\mathbb A^1}^{\mC Q}$, $\iota^{\mC Q}$ is inclusion,
and $\lambda_R^\mC Q$ is induced by the \K-injective resolution functor.
\end{thm}

The authors would like to thank Jorge Vit\'oria for plenty of helpful discussions
about recollements of Abelian and triangulated categories.

\section{Enriched Category Theory}\label{putrya}

In this section we collect basic facts about enriched categories we
shall need later. We refer the reader to~\cite{Borceux} for details.
Throughout this paper the quadruple $(\mathcal{V},\otimes,\underline{\Hom},e)$ is
a closed symmetric monoidal category with monoidal product
$\otimes$, internal Hom-object $\underline{\Hom}$ and monoidal unit
$e$. We sometimes write $[a,b]$ to denote $\underline{\Hom}(a,b)$,
where $a,b\in\Ob\cc V$. We have structure isomorphisms
   $$a_{abc}:(a\otimes b)\otimes c \to a\otimes (b\otimes c),\quad l_a:e\otimes a\to a,\quad r_a:a\otimes e\to a$$
in $\cc V$ with $a,b,c\in\Ob\cc V$.

\begin{definition}
A {\it $\mathcal{V}$-category}  $\mathcal{C}$, or {\it a category
enriched over $\mathcal{V}$}, consists of the following data:
\begin{enumerate}
\item a class $\Ob\mathcal{(C)}$ of objects;
\item for every pair $a,b \in$ $\Ob\mathcal{(C)}$ of objects, an object  $\mathcal{V}_{\cc C}(a,b)$ of
$\mathcal{V}$;
\item for every triple $a,b,c \in$ $\Ob\mathcal{(C)}$ of objects, a composition morphism in
  $\mathcal{V}$,
 $$c_{abc}:\mathcal{V}_{\cc C}(a,b) \otimes  \mathcal{V}_{\cc C}(b,c) \to  \mathcal{V}_{\cc C}(a,c);$$
\item for every object $a \in \mathcal{C}$, a unit morphism $u_a:e\to\mathcal{V}_{\cc C}(a,a)$ in
$\mathcal{V}$.
\end{enumerate}
These data must satisfy the natural associativity and unit axioms.

When $\Ob\cc C$ is a set, the $\cc V$-category $\cc C$
is called a {\it small $\cc V$-category}. 
\end{definition}

\begin{definition} \label{functor}
Given $\mathcal{V}$-categories $\mathcal{A},\mathcal{B}$, a {\it
$\mathcal{V}$-functor\/} or an {\it enriched functor\/}
$F:\mathcal{A} \to \mathcal{B}$ consists in giving:
\begin{enumerate}
\item for every object $a \in \mathcal{A}$, an object $F(a) \in
\mathcal{B}$;
\item for every pair $a,b \in \mathcal{A}$ of objects, a morphism in  $\mathcal{V},$
$$F_{ab}:\mathcal{V}_{\cc A}(a,b) \to  \mathcal{V}_{\cc B}(F(a),F(b))$$
in such a way that the following axioms hold:
\begin{enumerate}
\item[$\diamond$] for all objects $a,a',a''\in\cc A$, diagram~\eqref{B6.11} below commutes (composition
axiom);
\item[$\diamond$] for every object $a\in\cc A$, diagram~\eqref{fun} below commutes (unit
axiom).
\end{enumerate}
\begin{equation}    \label{B6.11}
\xymatrix{ \cc V_\cc A(a,a')\otimes\cc V_\cc
A(a',a'')\ar[rr]^(.6){c_{aa'a''}}\ar[d]_{F_{aa'}\otimes F_{a'a''}}
&&\cc V_\cc A(a,a'')\ar[d]^{F_{aa''}}\\
\cc V_\cc B(Fa,Fa')\otimes\cc V_\cc
B(Fa',Fa'')\ar[rr]_(.6){c_{Fa,Fa',Fa''}}&&\cc V_\cc B(Fa,Fa'') }
\end{equation}
\begin{equation}    \label{fun}
\xymatrix{
e\ar[r]^(.4){u_a}\ar[dr]_{u_{Fa}}&\cc V_\cc A(a,a)\ar[d]^{F_{aa}}\\
&\cc V_\cc B(Fa,Fa) }
\end{equation}
\end{enumerate}
\end{definition}

\begin{definition}\label{trans}
Let $\cc A, \cc B$ be two $\cc V$-categories and $F,G:\cc A\to \cc
B$ two $\cc V$-functors. A {\it $\cc V$-natural transformation\/}
$\alpha:F \Rightarrow G$ consists in giving, for every object
$a\in\cc A$, a morphism
    $$\alpha_a:e \to \cc V_\cc B(F(a),G(a))$$
in $\cc V$ such that diagram below commutes, for all
objects $a,a'\in\cc A$.
\begin{equation*}    \label{B6.13}
\xymatrix{
&\cc V_{\cc A}(a,a')\ar[ddl]_{l^{-1}_{\cc V_{\cc A}(a,a')}}\ar[ddr]^{r^{-1}_{\cc V_{\cc A}(a,a')}}&\\
&&\\
e\otimes\cc V_{\cc A}(a,a')\ar[d]_{\alpha_a \otimes G_{aa'}}&
&\cc V_{\cc A}(a,a')\otimes e\ar[d]^{F_{aa'}\otimes \alpha_{a'}}\\
\cc V_{\cc B}(Fa,Ga)\otimes\cc V_{\cc
B}(Ga,Ga')\ar[ddr]_{c_{FaGaGa'}\text{ }}
&&\cc V_{\cc B}(Fa,Fa')\otimes\cc V_{\cc B}(Fa',Ga')\ar[ddl]^(.45){\text{       }\text{  }  c_{FaFa'Ga'}}\\
&&\\
&\cc V_{\cc B}(Fa,Ga')& }
\end{equation*}
\end{definition}

Any $\cc V$-category $\cc C$ defines an ordinary category $\it{\cc{UC}}$, also called the {\it
underlying category}. Its class of objects is $\Ob \cc C$, the
morphism sets are $\Hom_{{\cc{UC}}}{(a,b)}:=\Hom_{\cc
V}(e, \cc{V_C}{(a,b))}$ (see~\cite[p.~316]{Borceux}).

Let $\cc C, \cc D$ be two $\cc V$-categories. The {\it monoidal
product\/} $\cc C\otimes \cc D$ is the $\cc V$-category, where
   $$\Ob(\cc C\otimes \cc D):=\Ob\cc C\times\Ob\cc D$$
and
   $$\cc V_{\cc C\otimes \cc D}((a, x),(b, y)):=\cc V_{\cc C}(a, b)\otimes\cc V_{\cc D}(x, y),\quad a,b\in\cc C,x,y\in\cc D.$$

\begin{definition}
A $\cc V$-category $\cc C$ is a {\it right $\cc V$-module} if there
is a $\cc V$-functor $\act:\cc C\otimes \cc V \to \cc C$, denoted
$(c,A) \mapsto c\oslash A$ and a $\cc V$-natural unit isomorphism
$r_c:\act(c,e)\to c$ subject to the following conditions:
\begin{enumerate}
\item there are coherent natural associativity isomorphisms $c\oslash (A \otimes B) \to (c \oslash A) \otimes
B$;
\item the isomorphisms $c\oslash (e \otimes A)\rightrightarrows c \oslash A$ coincide.

\end{enumerate}

A right $\cc V$-module is {\it closed\/} if there is a $\cc
V$-functor
$$\coact:\cc V^{\op} \otimes \cc C \to \cc C$$
such that for all $A \in \Ob \cc V$, and $c \in \Ob\cc C$, the $\cc
V$-functor $\act(-,A):\cc C \to \cc C$ is left
 $\cc V$-adjoint to $\coact(A,-)$ and $\act(c,-):\cc V \to \cc C$
is left $\cc V$-adjoint to $\cc V_{\cc C}(c,-)$. 
\end{definition}

If $\cc C$ is a small $\cc V$-category, $\cc V$-functors from $\cc
C$ to $\cc V$ and their $\cc V$-natural transformations form the
category $[\cc C,\cc V]$ of $\cc V$-functors from $\cc C$ to $\cc
V$. If $\cc V$ is complete, then $[\cc C,\cc V]$ is also a $\cc
V$-category whose morphism $\cc
V$-object $\cc V_{[\cc C,\cc V]}(X,Y)$ is the end
   \begin{equation*}\label{theend}
    \int_{\Ob \cc C} \cc V (X(c),Y(c)).
   \end{equation*}
   
\begin{lemma}\label{closedmod}
Let $\cc V$ be a complete closed symmetric monoidal category, and 
$\cc C$ be a small $\cc V$-category. Then $[\cc C,\cc V]$ is a closed $\cc V$-module.
\end{lemma}

\begin{proof}
See~\cite[2.4]{DRO}.
\end{proof}

Given $c \in \Ob \cc C$, $X\mapsto X(c)$ defines the $\cc V$-functor
$\text{Ev}_c:[\cc C,\cc V]\to\cc V$ called {\it evaluation at c}. The
assignment $c\mapsto\cc V_\cc C (c,-)$ from $\cc C$ to $[\cc C,\cc V]$ is
again a $\cc V$-functor $\cc C^{\op}\to[\cc C,\cc V]$, called the {\it $\cc
V$-Yoneda embedding}. $\cc V_\cc C (c,-)$ is a representable
functor, represented by $c$.

\begin{lemma}[The Enriched Yoneda Lemma]\label{enryon}
Let $\cc V$ be a complete closed symmetric monoidal category and
$\cc C$ a small $\cc V$-category. For every $\cc V$-functor $X:\cc C
\to \cc V$ and every $c\in \Ob \cc C$, there is a $\cc V$-natural
isomorphism $X(c) \cong \cc V_\cc F (\cc V_\cc C (c,-),X)$.
\end{lemma}

\begin{lemma}\label{bicomplete}
If $\cc V$ is a bicomplete closed symmetric monoidal category and
$\cc C$ is a small $\cc V$-category, then $[\cc C,\cc V]$ is
bicomplete. (Co)limits are formed pointwise.
\end{lemma}

\begin{proof}
See~\cite[6.6.17]{Borceux}.
\end{proof}

\begin{corollary}\label{polezno}
Assume $\cc V$ is bicomplete, and let $\cc C$ be a small $\cc
V$-category. Then any $\cc V$-functor $X:\cc C\to\cc V$ is $\cc
V$-naturally isomorphic to the coend
    $$X\cong\int^{\Ob\cc C}\cc V_\cc C(c,-)\oslash X(c).$$
\end{corollary}

\begin{proof}
See~\cite[6.6.18]{Borceux}.
\end{proof}

\section{Recollements for Grothendieck categories of enriched functors}

An \emph{adjoint triple} is a collection of functors $F,H:\mC C \to
\mC D$ and $G:\mC D \to \mC C$ such that $F$ is left adjoint to $G$ and $G$ is left adjoint to
$H.$ We denote this information $F\dashv G\dashv H.$ 

The following lemma is proven in~\cite[7.4.1]{Mac2}.

\begin{lemma}\label{mm}
Given an adjoint triple of functors $F\dashv G \dashv H$,
$F$ is fully faithful if and only if $H$ is fully faithful.
\end{lemma}

\begin{definition}
A \emph{recollement of an Abelian category \mC A by Abelian
categories \mC B and \mC C\/} is a diagram of additive functors
\begin{diagram*}[column sep=huge]
  \mC B\arrow[r,"i"]  
&\mC A
   \bendR{i_R}
   \bendL{i_L}
   \arrow[r,"r"]			
&\mC C
   \bendR{r_R}
   \bendL{r_L} 
\end{diagram*}
such that $i_L \dashv i \dashv i_R$ , $r_L \dashv r \dashv r_R$ are adjoint
triples, the functors $i, r_L$ and $r_R$ are fully faithful and $\Im i = \Ker r.$
\end{definition}

We refer the reader to~\cite[2.8]{PV} for the following proposition.

\begin{proposition}\label{adjseq}
Given a recollement of Abelian categories as above, 
for all $A\in \mC A$ there are $B,B'\in B$ such that the following  sequences $$0\to iB\to
r_L\circ rA\to A \to i\circ i_L A\to 0,$$ $$0\to i\circ i_R A\to A\to r_R\circ r A\to i B'\to 0$$ 
induced by the unit and
counit morphisms of the adjunctions are exact in \mC A. Furthermore, $\mC B$ is a Serre
subcategory of $\mC A$ and $r$ is naturally equivalent to the quotient functor
$\mC A\to\mC A/\mC B$. In particular, $\mC C\cong\mC A/\mC B$.
\end{proposition}

In what follows we fix a closed symmetric monoidal Grothendieck
category \mC V with a set of generators $\{g_i\}_{i\in I}$, and we also choose \mC C to be a small \mC V-category.
By~\cite[4.2]{AlGG} the category of enriched functors $[\mC C,\mC V]$
is Grothendieck with the set of generators $\{\mC V_{\mC C}(c,-)\oslash g_i\mid c\in\mC C,i\in I\}$.

Consider a full \mC V-subcategory $\mC A$ of the \mC V-category \mC C. This means 
that $\mC A$ is some subcollection of objects of \mC
 C and \mC V-morphism objects $\mC V_\mC
A (a,b)=\mC V_\mC C (a,b)$ for all objects $a,b\in \mC A$. 
We shall consider the full subcategory  $\mC S_\mC A := \{ F \in [\mC C, \mC V]
\mid F(a) = 0 \text{ for all } a\in \mC A\}$. Obviously, the full subcategory $\mC S_\mC A$ is localizing.

The following result fits $\mC S_\mC A$ and
the Grothendieck categories of enriched functors $[\mC C, \mC V]$, $[\mC A, \mC V]$ in a recollement.

\begin{theorem}\label{Egg}
Suppose $\mC A$ is a full \mC V-subcategory of a small \mC V-category \mC C.
Then there is a recollement \begin{diagram*}[column sep=huge]
\mC S_\mC A\arrow[r,"i"] &{[\mC C, \mC V]}
   \bendR{i_R}
   \bendL{i_L}
   \arrow[r,"r"]			
& {[\mC A, \mC V]}\bendR{r_R} 
   \bendL{r_L} 
\end{diagram*}
with functors $i,r$ being the canonical inclusion and restriction functors
respectively. The functors $r_L,r_R$ are the enriched left and right Kan extensions respectively,
$i_R$ is the torsion functor associated with the localizing subcategory $\mC S_\mC A$.
\end{theorem}

\begin{proof}
It is clear that $i$ is fully faithful and $\Im(i) = \Ker(r)$, as $r$ is the
same as precomposition with the inclusion $I: \mC A \hookrightarrow \mC C$. 
We shall construct each functor in turn, beginning with $r_R: [\mC A,\mC V]\to [\mC C,\mC
V].$ Define the functor $r_R$ as the enriched right Kan extension $$H \mapsto
{\int_{b\in\mC A}[\mC V_\mC C(-,Ib),H(b)]}.$$ This is well defined as \mC C is
small and \mC V is bicomplete. Further, we see that $r_RH(Ia)=H(a)$ for all
$a\in\mC A$, by the fully faithfulness of $I$ and the Yoneda lemma. We establish
that this is indeed right adjoint, or more precisely, there exists the 
following natural isomorphism 
   $$\mC \Hom_{[\mC A, \mC V]}(G\circ I,H) \cong \Hom_{[\mC C, \mC V]}(G, r_R(H)).$$ 
We shall prove it in the enriched case, from which the desired
isomorphism will follow. Consider the \mC V-morphism object 
   $$\mC V_{[\mC C, \mC V]}(G, r_R (H)) :=\int_{c\in \mC C} \mC [G(c),r_RH(c)]\cong
       \int_{c\in \mC C}\Big[G(c),\int_{b\in\mC A}[\mC V_\mC C(c,Ib),H(b)]\Big].$$
As the internal Hom of \mC V preserves \mC V-ends and the fact that we may swap limits, 
we see that this is naturally isomorphic to
   $$\int_{b\in\mC A}\int_{c\in \mC C}\Big[G(c),[\mC V_\mC C(c,Ib),H(b)]\Big].$$ 
Next, we use that \mC V is closed, and the limit preserving properties again to 
show that the latter object is isomorphic to
   $$\int_{b\in\mC A}\int_{c\in \mC C}\Big[\mC
       V_\mC C(c,Ib)\oslash G(c),H(b)\Big]\cong \int_{b\in\mC A}\Bigg[\int^{c\in \mC
       C}\mC V_\mC C(c,Ib)\oslash G(c),H(b)\Bigg].$$ 
This object is naturally isomorphic to 
   $$\int_{b\in\mC A}[G(Ib),H(b)]=\mC V_{[\mC A, \mC V]}(G\circ I, H).$$
Hence we can deduce the desired adjunction. 

We define the enriched left Kan extension $r_L:[\mC A,\mC V]\to [\mC C,\mC V]$ to be the
functor that acts on objects as 
   $$F\mapsto {\int^{a\in\mC A} \mC V_\mC C(Ia,-)\oslash F(a)}.$$ 
It is left adjoint to $r$ by~\cite[5.2]{AlGG}.
Having demonstrated that $ r_L \dashv r \dashv r_R,$ it is enough to know that either $r_L$ 
or $r_R$ is fully faithful to determine the other is also fully faithful (see Lemma~\ref{mm}).
We prove the full and faithfulness of $r_R$ as an enriched functor with the
following natural isomorphisms
\begin{align*}
\mC V_{[\mC C, \mC V]}(r_RX,r_RY)&\cong \mC V_{[\mC A, \mC V]}(r_R(X)\circ
I,Y)\phantom{XXXXXXXXXXxXXXXX}\\
&\cong \mC V_{[\mC A, \mC V]}\Bigg({\int_{b\in\mC A}[\mC V_\mC
C(-,Ib),X(b)]}\circ I,Y\Bigg)\\
&\cong \mC V_{[\mC A, \mC V]}\Bigg({\int_{b\in\mC A}[\mC V_\mC
C(I(-),Ib),X(b)]},Y\Bigg).
\end{align*}
Since $I$ is a fully faithful enriched functor, we have a natural
isomorphism to
\begin{align*}
&\cong \mC V_{[\mC A, \mC V]}\Bigg({\int_{b\in\mC A}[\mC V_\mC
A(-,b),X(b)]},Y\Bigg)\\
&= \int_{a\in\mC A}{\Bigg[\int_{b\in\mC A}[\mC V_\mC A(a,b),X(b)]},Y(a)\Bigg]\\
&\cong \int_{a\in\mC A}{\Big[\mC V_{[\mC A,\mC V]}(\mC V_\mC
A(a,-),X)},Y(a)\Big]\\
&\cong \int_{a\in\mC A}[X(a),Y(a)]=:\mC V_{[\mC A, \mC V]}(X,Y).\\
\end{align*}
Thus we conclude that both $r_L,r_R$ are fully faithful.	

Since $\mC S_\mC A$ is a localizing subcategory of the Grothendieck
category $[\mC C, \mC V]$, its inclusion always has a right adjoint. This right
adjoint is simply the $\mC S_\mC A$-torsion functor. However, it is known that
the adjoint triple $r_L\dashv r\dashv r_R$ can be used to define the
adjoint functors $i_L, i_R$ explicitly. For completeness, we demonstrate this.

First, given that
$r_L\dashv r$, denote the counit of this adjunction by $\eps :r_L
\circ r \Rightarrow \id,$ and the unit $\mu:\id \To r\circ r_L$.
Given the triangle identities, namely that $r(\eps_X) \circ \mu_{r(X)}=\id_X$,
we have that $r(\eps_X)$ is an epimorphism, and with $r$ exact, it follows that
$r(\Coker(\eps_X))\cong \Coker(r(\eps_X))\cong 0.$ Thus $\Coker(\eps_X)$ belongs 
to $\Ker(r)=
\mC S_\mC A$ for all $X \in [ \mC C,\mC V].$ Hence we may define a functor
   $$i_L:[\mC C,\mC V] \to \mC S_\mC A,\quad X \mapsto \Coker(\eps_X).$$ 
In order to show that this is indeed adjoint, we consider the canonical map  $X\to
\Coker(\eps_X)$. Take any map $f:X\to Y$ for some object $Y \in \mC S_\mC A.$
Since $Y\in \mC S_\mC A$ we have that $\eps_Y:0\to Y,$ then by naturality we see
that $f\circ\eps_X= 0.$ Thus there exists a unique map $\Coker(\eps_X)\to Y$ that factors $f$. Therefore we
conclude that $i_L$ is left adjoint to $i$. Conversely, we define the right
adjoint $i_R: Y \mapsto \Ker(\eta_Y)$ where $\eta$ is the unit of the
adjunction $r \dashv r_R,$ the proof is entirely dual. The construction of the desired recollement is completed.
\end{proof}
  
The preceding theorem and Proposition~\ref{adjseq} imply $[\mC A, \mC V]\cong [\mC C,\mC
V]/\mC S_\mC A.$ An equivalence of these categories was also proven in~\cite[5.3]{AlGG} by using the following
explicit functors. Denote by $\ell:[\mC C,\mC V]\to[\mC C,\mC V]/\mC S_\mC A$ the $\mC S_{\mC A}$-localization functor.
Then the composite functor
   \begin{equation}\label{kappa}
    \varkappa:=\ell\circ r_L:[\mC A,\mC V]\to[\mC C,\mC V]/\mC S_\mC A
   \end{equation}
is an equivalence of Grothendieck categories by~\cite[5.3]{AlGG}.

\begin{theorem}\label{Egg2}
Suppose $\mC A$ is a full \mC V-subcategory of a small \mC V-category \mC C.
Then there is a recollement \begin{diagram*}[column sep=huge]
\mC S_\mC A\arrow[r,"i"] &{[\mC C, \mC V]}
   \bendR{i_R}
   \bendL{i_L}
   \arrow[r,"\ell"]			
& {[\mC C,\mC V]/\mC S_\mC A}\bendR{\ell_R} 
   \bendL{\ell_L} 
\end{diagram*}
with functors $i_L,i,i_R$ being from Theorem~\ref{Egg}. The functor $\ell_R$ is the inclusion and
$\ell_L:=r_L\circ\varkappa^{-1}$.
\end{theorem}

\begin{proof}
One has $\Im(i) = \Ker(\ell)$ and $\ell_R$ is a right adjoint of $\ell$ by the general localization theory
of Grothendieck categories. To see that $\ell_L$ is left adjoint to $\ell$, let $X\in[\mC C,\mC V]/\mC S_\mC A$
and $Y\in[\mC C,\mC V].$ The proof of~\cite[5.3]{AlGG} shows that the adjunction morphism 
$r_Lr(Y)\to Y$ induces an isomorphism $\ell r_Lr(Y)\cong\ell(Y)$. Then
   $$\Hom_{[\mC C,\mC V]}(\ell_L(X),Y)=\Hom_{[\mC C,\mC V]}(r_L\circ\varkappa^{-1}(X),Y)
       \cong\Hom_{[\mC C,\mC V]/\mC S_{\mC A}}(X,\varkappa r(Y)).$$
But $\varkappa r(Y)=\ell r_Lr(Y)\cong \ell(Y)$. It follows that
   $$\Hom_{[\mC C,\mC V]/\mC S_{\mC A}}(X,\varkappa r(Y))\cong\Hom_{[\mC C,\mC V]/\mC S_{\mC A}}(X,\ell(Y)),$$
and hence $\ell_L$ is left adjoint to $\ell$. Since $\ell_R$ is fully faithful then so is $\ell_L$ by Lemma~\ref{mm}.
\end{proof}

\begin{example}\label{EGG}
Theorem~\ref{Egg} recovers an important example
of a recollement\cite[Chapter~IV]{Stein}. First, consider a ring $R$ and assume there
exists some idempotent $e=e^2\in R.$ Then there exists a recollement  

\begin{diagram*}[column sep=huge]
R/ReR\textrm{-Mod}\arrow[r] &R\textrm{-Mod}
   \bendR{}
   \bendL{}
   \arrow[r]			
& eRe\textrm{-Mod}.
\bendR{}
\bendL{}
\end{diagram*}
To derive this well known example from Theorem~\ref{Egg},
 we shall need to define a further preadditive category $\mC R$ with two objects
  $\mC E$, ${\mC E}^*.$ The morphism groups are given by
  \begin{diagram*}[row
  sep=tiny] \mC R (\mC E,\mC E):= eRe& \mC R (\mC E^*,\mC E):=eR(1-e)\\  
 \mC R (\mC E,\mC E^*):= (1-e)Re& \mC R (\mC E^*,\mC E^*):= (1-e)R(1-e).
\end{diagram*}
We define composition by multiplication in $R$. 
Set $\mC S_\mC E := \{ F \in [\mC R,
 \Ab] \mid F(\mC E) = 0\}\subset [\mC R,
 \Ab]$. Notice that $[\mC R,\Ab]$ is the ordinary category of additive functors from 
$\mC R$ to Abelian groups. By Theorem~\ref{Egg} one has the following recollement
  \begin{diagram*}[column sep=huge]
\mC S_\mC E\arrow[r,"i"] &{[\mC R,\Ab]}
   \bendR{i_R}
   \bendL{i_L}
   \arrow[r,"r"]			
& {[\mC E,\Ab].}
   \bendR{r_R}
   \bendL{r_L}
\end{diagram*}
We wish to relate this recollement to the module categories above.
We begin with the observation that $[\mC E,\Ab]$ is isomorphic to the
category of left $eRe$-modules, as $\End(\mC E)=eRe$, and write $M(\mC
E)=M\in\textrm{Mod-}eRe$.
Further,  $[\mC R,\Ab]$ is  equivalent to the category of left $R$-modules.
This follows from the fact that $[\mC R,\Ab]$ is equivalent to the category of
preadditive functors from a single object with morphism group the generalised matrix ring $\left( {\begin{array}{cc} eRe & eR(1-e) \\
(1-e)Re & (1-e)R(1-e)\end{array} } \right).$ 
The proof of this fact follows from a theorem of
Mitchell~\cite[Theorem 7.1]{Mitchell}. It says in particular that for a preadditive category $\mC C$ with finitely
many objects, a preadditive
functor $F:\mC C\to \Ab$ is equivalent to the module $\bigoplus_{X\in\mC C}F(X)$.
We also note that this matrix ring is the Peirce decomposition of $R$. Hence
the matrix ring is isomorphic to $R$ and $[\mC R,\Ab]$ is
equivalent to the category of left $R$-modules. Next, $\mC S_\mC E$ is a full 
subcategory of $[\mC R,\Ab]$, so given any $S\in \mC S_\mC E$ we have a canonical isomorphism
    $$ S \cong
{\int_{Y\in\mC R}\Hom_{\Ab}\big(\mC R(-,Y), S(Y)\big)}.$$  
  This is a restatement of the enriched Yoneda lemma. If we require that
  $S(\mC E)=0$ then this is the same as
$$0=S(\mC E)\cong\int_{Y\in\mC R}\Hom_{\Ab}\big(\mC R(\mC E,Y),
  S(Y)\big) \cong \Hom_{[\mC R,\Ab]}(\mC R(\mC E,-),S).$$
As an $R$-module, ${\mC
  R}(\mC E,-)$ is equivalent to ${\mC
  R}(\mC E,\mC E)\oplus{\mC
  R}(\mC E,\mC E^*)=eRe\oplus (1-e)Re=Re.$ Thus we derive that $S\in\mC S_\mC E$
  if and only if $$\Hom_R(Re,S)\cong\Hom_{[\mC R,\Ab]}(\mC R(\mC E,-),S)\cong
  0.$$ 
$\mC S_\mC E$ is equivalent to the subcategory of
$R$-modules $S$ with $\Hom_R(Re,S)=eS=0$. This subcategory can be identified with
the category of $R/ReR$-modules~\cite[p.~42]{Stein}.
   
We calculate $r_L,r,r_R$ as follows. We use the enriched Yoneda lemma again,
writing $I$ for the inclusion $\mC E\hookrightarrow \mC R$ and derive that
  $$r:X\mapsto X\circ I\cong{\int_{Y\in\mC R}\Hom_{\Ab}\big(\mC R(I(-),Y), X(Y)\big)}.$$ 
As an $eRe$-module it is isomorphic to
$\Hom_R(Re,X)$, and so $r=\Hom_R(Re,-)$ in terms of modules. Thus we are able
  to identify our gluing with the well known gluing (see, e.g.,~\cite[4]{Stein} and~\cite[2.7]{P})
\begin{diagram*}[column sep=huge]
{R/ReR}\textrm{-Mod}\arrow[r,"{inclusion}"] &{R}\textrm{-Mod}
\bendR{\Hom_R (R/ReR,-)}
   \bendL{R/ReR\tens_{R}-}
   \arrow[rr,"{\Hom_R (Re,-)}"]			
& &{eRe}\textrm{-Mod}.
\arrow[ll,"{Re\tens_{eRe}-}" above,shift right, end
anchor=north east,bend right, start anchor=north west]
\arrow[ll,"{\Hom_{eRe}(Re,-)}" below,shift left,end
anchor=south east,bend left, start anchor=south west]
\end{diagram*} 
\end{example}

Our next goal is to extend recollements of Theorems~\ref{Egg} and~\ref{Egg2}
to triangulated recollements of associated derived categories.

\section{Recollements for the derived categories of enriched functors}

We refer the reader to~\cite{Spal} for the notions and basic properties of $\K$-projective and $\K$-injective resolutions.
In what follows we always assume that the following theorem is satisfied, which was proven by the authors in~\cite[6.2]{GGME}.

\begin{theorem}\label{cond}
Let (\,$\mC V\!, \tens,e$) be a closed symmetric monoidal Grothendieck category
such that the derived category of \mC V  is a compactly
generated triangulated category with compact generators $\{P_j\}_{j\in J}$. Further,
suppose we have a small \mC V-category \mC C and that any one of the following
conditions is satisfied:
\begin{itemize}[leftmargin=1.4em]
  \item[1.] each $P_j$ is $\K$-projective;
  \item[2.] for every $\K$-injective $Y\!\in\! \Ch [\mC C,\mC V]$ and every $c\in
  \mC C$, the complex $Y(c)\in \Ch(\mC V)$ is $\K$-injective;
  \item[3.] $\Ch(\mC V)$ has a model structure, with quasi-isomorphisms
  being weak equivalences, such that for every injective fibrant
  complex $Y\in \Ch [\mC C,\mC V]$ the complex $Y(c)$ is fibrant in $\Ch(\mC V)$.
\end{itemize}
Then $\D[\mC C,\mC V]$ is a compactly generated triangulated category with
compact generators $\{\mC V_{\mC C}(c,-)\oslash Q_j\mid c\in \mC C, j\in J\}$ where, if we
assume either (1) or (2), $Q_j = P_j$ or if we assume (3) then $Q_j = P^c_j$
a cofibrant replacement of $P_j$.
\end{theorem}

\begin{lemma}\label{rest}
Under the assumptions of Theorem~\ref{cond} let $\mC A$ be a full \mC V-subcat\-egory of $\mC C$. Then
$\Ch[\mC A, \mC V]$ also satisfies either condition (1), (2) or (3). In particular,
$\D[\mC A,\mC V]$ is a compactly generated triangulated category with
compact generators $\{\mC V_{\mC A}(a,-)\oslash Q_j\mid a\in \mC A, j\in J\}$ where, if we
assume either (1) or (2), $Q_j = P_j$ or if we assume (3) then $Q_j = P^c_j$
a cofibrant replacement of $P_j$.
\end{lemma}

\begin{proof}
Assume (1) holds for $\Ch[\mC C,\mC V]$. As this is a condition solely reliant
on $\D(\mC V)$ it holds automatically for $\Ch[\mC A,\mC V]$ as well.

Assume (2) holds for  $\Ch[\mC C, \mC V]$.
Given any \K-injective $ Y\in\Ch[\mC A, \mC V]$, consider $r_RY$ in $\Ch[\mC
C, \mC V]$. Given an acyclic complex $X \in \Ch[\mC C, \mC V]$, $rX$ is acyclic
in $\Ch[\mC A, \mC V]$. We see that 
   $$\K[\mC C, \mC V](X,r_RY)\cong \K[\mC A, \mC V](rX,Y)\cong 0.$$
Thus $r_RY$  is \K-injective in
$\Ch[\mC C, \mC V]$. 
By assumption for all $c\in \mC C$ we have $r_RY(c)$ is \K-injective in
$\Ch(\mC V)$. Since $r_RY(a)=Y(a)$ on all $a \in\mC A$ at every degree (see
the proof of Theorem~\ref{Egg}), it follows that $Y(a)$ is \K-injective.
We deduce that (2) holds for $\Ch[\mC A,\mC V]$.

Assume (3) holds for  $\Ch[\mC C, \mC V]$. Consider an injective fibrant
$Y\in\Ch[\mC A, \mC V]$.
Given $X, X^\prime \in \Ch[\mC C, \mC V]$ and an
injective quasi-isomorphism $\alpha:X \to X^\prime$, let
$f:X\to r_RY$ be any map. Then  $r_RY$ is injective fibrant if and only if there
exists a lift $X^\prime\to r_RY$ that factors $f$ in $\Ch[\mC C, \mC
V]$. Consider the following diagrams
\begin{diagram*}
X\ar[r,"f"]\ar[d,"\alpha" swap]&r_RY
\arrow[drr,phantom,"\overset{}{\Longleftrightarrow}" description] &  &
rX\ar[r,"\phi(f)"]\ar[d,"r\alpha" swap]&Y\\
X^\prime&\phantom{t} & 
&rX^\prime\ar[ru,"\exists h",dashed,swap]&\phantom{t}
\end{diagram*} 
where $\phi(f)$ is adjoint to $f$. We see that the existence of a map $X'\to r_RY$ extending $f$ 
follows from the existence of $h$ on the right.
Since $r$ is exact, it takes injective quasi-isomorphisms to injective quasi-isomorphisms. 
Further, $Y$ is injective fibrant in $\Ch[\mC A, \mC V]$ by assumption.
Therefore a lift $h:r(X^\prime)\to Y$ exists. Hence we
have a lift $X^\prime\to r_RY$ and deduce $r_R Y$ is
 injective fibrant in $\Ch[\mC C, \mC V]$.
By assumption, it follows that $r_R Y(c)$ is fibrant in
$\Ch(\mC V)$ for all $c\in \mC C$. In particular, for all $a\in \mC
A,$ we see that $Y(a)=r_RY(a)$ is fibrant. Consequently (3) holds for $\Ch[\mC A, \mC V]$.
\end{proof}

\begin{definition}
A \emph{recollement of a triangulated category \mC D by triangulated
 categories \mC B and \mC C\/} is a diagram of triangulated functors
 \begin{diagram*}[column sep=huge]
  \mC B\arrow[r,"\iota"]  
&\mC A
   \bendR{\iota_R}
   \bendL{\iota_L}
   \arrow[r,"\rho"]			
&\mC C
   \bendR{\rho_R}
   \bendL{\rho_L} 
\end{diagram*}
such that $\iota_L \dashv \iota \dashv \iota_R$, $\rho_L \dashv \rho \dashv
\rho_R$ are adjoint triples, the functors $\iota, \rho_L$ and $\rho_R$ are fully
faithful and $\Im \iota = \Ker \rho.$
\end{definition}

\begin{proposition}\label{adjtri}
Given a recollement of triangulated categories as above, 
 for all $A\in \mC A$ we have triangles $$\rho_L\rho
A\to A \to \iota \iota_L A\to (\rho_L\rho A)[1]\text{,\quad  and\quad}\iota
\iota_R A\to A\to \rho_R\rho A\to(\iota \iota_R
A)[1]$$ induced by the unit and counit morphisms of the adjunctions in \mC A.
\end{proposition}

\begin{proof}
This follows from~\cite[7.3(ii)]{P}.
\end{proof}

\begin{definition}
Define the following full subcategories of $\D[\mC C,\mC V]$. Given a full \mC
V-subcategory $\mC A\subset \mC C$, denote by  $\mC T_\mC A$ the smallest localizing subcategory
of $\D[\mC C,\mC V]$ generated by the subcollection of compact generators 
   $$\mC T_\mC A :=\langle \mC V_\mC C(a,-)\oslash
       Q_j \mid a\in \mC A, Q_j \in \D(\mC V)\rangle,$$ 
where $Q_j$ are those compact
generators of $\D (\mC V)$ as formulated in Theorem~\ref{cond}.
Define a further full subcategory of $\D[\mC C,\mC V]$ as 
   $$\mC E_\mC A := \{Y \in \mC D[\mC C, \mC V] \mid
       H_n(Y(a)) = 0 \text{ for all } a\in \mC A\text{ and }n \in \bB Z\}.$$
\end{definition}

\begin{lemma}\label{perp}
$\mC T_\mC A ^\perp = \mC E_\mC A. $
\end{lemma}

\begin{proof}
By the proof of~\cite[6.2]{GGME} there is a natural isomorphism 
\begin{align}\label{tenshom}
\Hom_{\D [\mC C, \mC V]}(\mC V_\mC
C(a,-)\oslash Q_j, Y) \cong \Hom_{\D(\mC V)}( Q_j, Y(a)).
\end{align} 
Since $Q_j$-s are compact generators of $\D(\mC V)$, it follows that
$\Hom_{\D [\mC C, \mC V]}(\mC V_\mC 
C(a,-)\oslash Q_j, Y)= 0$ for all $Q_j$-s if and only if $Y(a)$ is acyclic. 
Thus if we assume that $Y$ belongs to $\mC
E_\mC A$, then $Y(a)=0$ for all $a\in \mC A$, and hence
$\Hom_{\D [\mC C, \mC V]}(\mC V_\mC C(a,-)\oslash Q_j, Y)= 0$ for all $Q_j$-s. Therefore  $\mC
E_\mC A \subset \mC T_\mC A ^\perp$, because by a theorem of Neeman~\cite[2.1]{Nee96} 
$\mC T_\mC A$ is the smallest triangulated full subcategory 
containing $\mC V_\mC C(a,-)\oslash Q_j$-s. Likewise, assuming $X\in\mC T_\mC A^\perp$ 
then $\Hom_{\D [\mC C, \mC V]}(\mC V_\mC 
C(a,-)\oslash Q_j, Y)= 0$ for all $a\in\mC A$ and $Q_j$-s. Since $Q_j$-s are
compact generators of $\D(\mC V)$, it follows from the isomorphism~\eqref{tenshom}
that $Y(a)$ is acyclic for all $a\in
\mC A$. Therefore $\mC T_\mC A ^\perp \subset\mC E_\mC A.$ 
We conclude that  $\mC T_\mC A ^\perp =\mC E_\mC A.$
\end{proof}

\begin{lemma}\label{comp}
There is a recollement of triangulated categories
\begin{diagram*}[column sep=huge]
\mC T_\mC A^\perp\arrow[r,"\iota"] &{\D[\mC C,\mC V]}
\bendR{\iota_R} 
\bendL{\iota_L}
\arrow[r,"\tau"]		
& \mC T_\mC A,
\bendR{\tau_R} 
\bendL{\tau_L}
\end{diagram*}
in which $\iota,\tau_L$ are inclusions.
\end{lemma}

\begin{proof}
This follows from~\cite[5.6.1]{Krause2} if we observe that the subcategory $\mC T_\mC A$ is 
compactly generated by a subcollection of compact generators.
\end{proof}

\begin{lemma}\label{equ}
Let \mC S and \mC T be compactly generated triangulated categories. Suppose
there exists a set of compact generators $\Sigma$ in \mC S and a triangulated functor 
$F:\mC S\to \mC T$ that preserves direct sums such that
\begin{itemize}
  \item[1.] the collection $\bigl\{F(X)|X\in \Sigma\bigr\}$ is a set of compact generators in \mC T\!,
  \item[2.] for any $X,Y$ in $\Sigma$, the induced map 
  $$F_{X,Y[n]}:\Hom_{\mC S} \bigl(X,Y[n]\bigr)\to \Hom_{\mC T}\bigl(FX,FY[n]\bigr)$$ is an isomorphism for all $n\in\bB Z$.
\end{itemize}
Then $F$ is an equivalence of triangulated categories.
\end{lemma}

\begin{proof} 
To show that $F$ is an equivalence, we need to demonstrate that it
is fully faithful and essentially surjective. Without loss of generality, we may
assume that $\Sigma$ is closed under direct summands 
and shifts. Denote the collection of compact objects 
belonging to $\mC S$ by $\mC S^c$. Then $\mC S^c$ is recovered as
$\bigcup_{n\geq 0} \Sigma_n$ where $\Sigma_0 = \Sigma$ and for $n> 0,$ we set
$\Sigma_n$ to consist of the direct summands of objects in
$\{Z\in \mC S^c| \textrm{ there exists a triangle } X\to Y\to Z \to X[1]\textrm{ with } X,Y\in \Sigma_{n-1}\}.$
We also use here Neeman's Theorem~\cite[2.1]{Nee96}.
We have that $F$ is fully faithful on $\Sigma_0$ by assumption. 

Using a long exact sequence and
the five lemma we also have that (2) holds for 
$X\in\Sigma_0,Y\in \Sigma_1$, similarly any $X\in\Sigma_1,Y\in\Sigma_0$.
Suppose $Z,Z'\in\Sigma_1,$ then there is a triangle $X\to Y\to Z\to X[1]$ with
$X,Y\in\Sigma_0$. Consider a commutative diagram
\begin{diagram*}(Z',X) \arrow[r]\arrow[d]& (Z',Y) \arrow[r]\arrow[d]&
(Z',Z) \arrow[r]\arrow[d,dashed]&(Z',X[1]) \arrow[r]\arrow[d]&\dots\\
(FZ',FX)\arrow[r] &  \arrow[r](FZ',FY)&  \arrow[r](FZ',FZ) &
\arrow[r](FZ',FX[1])&\dots
\end{diagram*} 
where the vertical arrows are those induced by the functor $F.$
Applying the five lemma, we see that the dashed arrow is an isomorphism and can
conclude that (2) holds for $\Sigma_1.$ Proceeding by induction, we see
that (2) holds for $\Sigma_n$, for all $n\in\bB Z,$ and further holds for
$\mC S^c.$ We also see that $F$ takes objects in $\mC S^c$ to compact objects in
$\mC T\!.$

Let $X\in\mC S^c$ and let $\mC S'$ be the full subcategory in \mC S of those objects $Z\in\mC S$
for which the homomorphism
   $$F_{X,Z[n]}:\Hom_{\mC S} \bigl(X,Z[n]\bigr)\to \Hom_{\mC T}\bigl(FX,FZ[n]\bigr)$$ 
is an isomorphism for all $n\in\bB Z$. We have shown that $\mC S^c\subset \mC S'.$ Using
the five lemma, one can show as above that $\mC S'$ is triangulated. We claim that $\mC S'$ is
closed under direct sums. Indeed, let $\{Z_i\}_I$ be a family of objects in $\mC
S'$. Consider the following commutative diagram
\begin{diagram*}
(X,\oplus_IZ_i)\arrow[d,dashed]\arrow[r,
"\cong"]&\arrow[d,"\cong"]\oplus_I(X,Z_i)\\
(FX,\oplus_IFZ_i)\arrow[r, "\cong"]&\oplus_I(FX,FZ_i).
\end{diagram*} 
Since $F$ preserves direct sums by assumption, 
the dashed arrow is an isomorphism and $\mC S'$ is closed under
direct sums. By Neeman's Theorem~\cite[2.1]{Nee96} this implies that $\mC S'=\mC S$. 

Now let $\mC S''$ be the full subcategory in \mC S of those objects $Y\in\mC S$
for which the map
   $$F_{Y,Z}:\Hom_{\mC S} (Y,Z)\to \Hom_{\mC T}(FY,FZ)$$ 
is an isomorphism for all $Z\in\mC S$. We 
have shown that $\mC S^c\subset \mC S''$. $\mC S''$ is plainly
closed under direct sums. Using
the five lemma, one can show as above that $\mC S'$ is triangulated.
By Neeman's Theorem~\cite[2.1]{Nee96} $\mC S''=\mC S$, and hence $F$ is fully faithful.

It remains to show that $F$ is essentially surjective. 
Let $\mC T'$ be the essential image of $F$. Then $\mC T^c\subset \mC T'$\!, because
$F(\mC S^c)$ is a full subcategory of compact generators and one can use the same induction
arguments as above to show that any compact objects in $\mC T$ is isomorphic to the
image of a compact object in $\mC S$. Clearly, $\mC T'$ is triangulated and closed under 
direct sums. By Neeman's Theorem~\cite[2.1]{Nee96} $\mC T'=\mC T$\!, and hence $F$ is an equivalence.
\end{proof}

Recall that $\D [\mC C,\mC V]$ is compactly generated by objects
$\bigl\{\mC V_\mC C(c,-)\oslash Q_j\mid c\in\mC C\bigr\}$ where $\{Q_j\}_J$ is a family
of compact generators of $\D(\mC V)$ (see Theorem~\ref{cond}). Given 
a collection of objects $\mC A\subset\mC C$, the exact restriction functor
$r:[\mC C,\mC V]\to[\mC A,\mC V]$ of Theorem~\ref{Egg} induces a 
triangulated functor $\rho:\D [\mC C,\mC V]\to\D[\mC A,\mC V]$ that applies
$r$ degreewise to any complex. By Lemma~\ref{rest} $\bigl\{\mC V_\mC A(a,-)\oslash Q_j\mid a\in\mC A,j\in J\bigr\}$
is a family of compact generators of $\D [\mC A,\mC V]$.

\begin{lemma}\label{L6}
For all $a\in \mC A$, $X\in \D [\mC C,\mC V]$ and $Q\in\{Q_j\}_J$, we have a
natural isomorphism $\rho:\Hom_{\D [\mC C,\mC V]}\bigl(\mC V_\mC C(a,-)\oslash Q,X\bigr)\xrightarrow{\cong}
\Hom_{\D[\mC A,\mC V]}\bigl(\mC V_\mC A(a,-)\oslash 
Q,\rho X\bigr)$ induced by the triangulated functor $\rho$.
\end{lemma}

\begin{proof} 
The proof of this lemma reduces to verifying commutativity of the following diagram
\begin{diagram*}
{{\Hom_{\D[\mC C,\mC V]}}\bigl(  
    \mC V_\mC C(a,-)\oslash Q_j,X\bigr) }
    \arrow[d,"\cong"]\arrow[r,"\rho"] & {\Hom_{\D[\mC A, \mC V]}\bigl( \mC V_\mC
    A(a,-)\oslash Q_j, \rho X\bigr)}\arrow[d,"\cong"]\\ 
\Hom_{\D(\mC V)}\bigl(Q_j,X(a)\bigr)\arrow[r,"="]& \Hom_{\D(\mC
V)}\bigl(Q_j,\rho X(a)\bigr),
\end{diagram*} 
where the vertical arrows are the isomorphisms~\eqref{tenshom}.

Consider a map 
\begin{diagram*} 
\alpha \circ q^{-1}:&\mC V_\mC C(a,-)\oslash Q_j  \arrow[dr,"\alpha" swap] & &
\arrow[dl,"q"]\phantom{\mC V_\mC A(} X\phantom{a,-)\oslash} \\& &Y&                                
\end{diagram*}  
belonging to $\Hom_{\D[\mC C,\mC V]}\bigl(\mC V_\mC C(a,-)\oslash Q_j, X\bigr),$
where $\alpha$ is a chain map and $q$ is a quasi-isomorphism. If we apply $\rho$, 
then we have a diagram
\begin{diagram*} 
\rho(\alpha \circ q^{-1}):&\mC V_\mC A(a,-)\oslash Q_j  \arrow[dr,"r(\alpha)"
swap] & & \arrow[dl,"r(q)"]\phantom{\mC V_\mC A(} \rho X\phantom{a,-)\oslash}
\\& &\rho Y& \end{diagram*} 
We consider the image of this morphism across the
adjunction. We see that $\rho(\alpha\circ q^{-1})$ is taken to a map
\begin{diagram*}
Q_j\arrow[dr,"\overline{r(\alpha)}"
swap] & & \arrow[dl,"r(q)_a"]\rho X(a)\\ &\rho Y(a)&                                
\end{diagram*} 
in $\D(\mC V),$ where $\overline {r (\alpha)}$ is the image of $r (\alpha)$
under the isomorphism~\eqref{tenshom}. On the other hand, consider the image of
$\alpha\circ q^{-1}$ under the isomorphism~\eqref{tenshom}. It is is mapped to 
\begin{diagram*} Q_j\arrow[dr,"\overline\alpha" swap] & &
\arrow[dl,"q_a"{name=q}]X(a) &Q_j\arrow[dr,"\overline\alpha"{name=alp,swap}
] & & \arrow[dl,"r(q)_a"]\rho X(a)\\ & Y(a)& & & \rho Y(a)&
\arrow["="{font = \Large},phantom,description,from=q,to=alp]\end{diagram*} 
Clearly, $r(q)_a=q_a,$ and we draw the following diagram
\begin{diagram*}[row sep=huge]
 Q_j\arrow[d,"u_\mC C" swap]\arrow[r,"="]&Q_j\arrow[d,"u_\mC A"] \\
\underline\Hom_{\Ch(\mC V)} \bigl(\mC V_\mC C(a,-),\mC V_\mC C(a,-)\oslash
Q_j\bigr)\arrow[d,"{\underline\Hom_{\Ch(\mC V)}(\mC V_\mC
C(a,-),\alpha)}" description]\arrow[r,"r"]& \arrow[d,"{\underline\Hom_{\Ch(\mC
V)}(\mC V_\mC A(a,-),r\alpha)}" description]\underline\Hom_{\Ch(\mC V)}\bigl(\mC
V_\mC A(a,-),\mC V_\mC A(a,-)\oslash Q_j\bigr)\\
 \underline\Hom_{\Ch(\mC V)}\bigl(\mC V_\mC C(a,-),Y\bigr)\arrow[d]\arrow[r,"r"]&  
 \arrow[d]\underline\Hom_{\Ch(\mC V)}\bigl(\mC V_\mC A(a,-),\rho Y\bigr)\\
 Y(a) \arrow[r,"="]& \rho Y(a)                                
\end{diagram*}
where the left most path is $\overline\alpha$ and the right most path is
$\overline{r(\alpha)}.$  We identify the canonical morphism
$u_\mC C$ associated with the identity on $\mC V_\mC C(a,-)\oslash Q_j,$ and
similarly we identify $u_\mC A$ with the canonical morphism associated with the identity on 
and $\mC V_\mC A(a,-)\oslash Q_j.$ The first square of the diagram commutes as
the unit morphisms $e\to \mC V_\mC C(a,a)$ and $e\to \mC V_\mC
A(a,a)$ are equal. The second square is obviously
commutative. The third square commutes by the Enriched
Yoneda Lemma. We conclude that $\overline{r(\alpha)}$ equals
$\overline\alpha.$ Hence the lemma.
\end{proof}

\begin{lemma}\label{F}
We define the following triangulated functor $F:\mC
T_{\mC A}\hookrightarrow \D[\mC C,\mC V]$ $\to \D[\mC A, \mC V]$ which consists of
the inclusion $\tau_L$ of Lemma~\ref{comp} followed by
$\rho:\D[\mC C, \mC V]\to \D[\mC A, \mC V]$.
Then $F:\mC T_{\mC A}\to\D[\mC A, \mC V]$ is  an equivalence of categories.
\end{lemma}

\begin{proof}
We use standard arguments regarding compactly generated triangulated categories.
By Lemma~\ref{rest} $\D[\mC A, \mC V]$ is compactly generated.

Following Lemma~\ref{equ} it is enough to show that $F$ sends compact
generators to compact generators and induces an isomorphism between their Hom-sets.
By definition, $\{\mC V_{\mC C}(a,-)\oslash Q_j\mid
a\in \mC A, j\in J\}$ is a family of compact generators of $\mC T_\mC A$. 
The functor $F$ maps $\mC V_{\mC 
C}(a,-)\oslash Q_j$ to $\mC V_{\mC A}(a,-)\oslash Q_j$, a compact generator
of $\D[\mC A, \mC V]$. Next, as $\tau_L$ is inclusion,
we need only verify whether the map
\footnotesize
\begin{diagram*}
{{\Hom_{\D[\mC C,\mC V]}}(  
    \mC V_\mC C(a,-)\oslash Q_j, \mC V_\mC C(b,-)\oslash Q_k)}
    \arrow[r,"\rho"] & {\Hom_{\D[\mC A, \mC V]}( \mC V_\mC
    A(a,-)\oslash Q_j,\mC V_\mC A(b,-)\oslash Q_k)}
\end{diagram*}
\normalsize is an isomorphism. This follows from Lemma~\ref{L6}.  We conclude
that $F$ is an equivalence by Lemma~\ref{equ}.
\end{proof}

The following result is an extension of the recollement of Theorem~\ref{Egg} to derived categories.

\begin{theorem}\label{TA}
Under the assumptions of Theorem~\ref{cond} there exists
a recollement of triangulated categories 
\begin{diagram*}[column sep=huge]
\phantom{X}\mC E_\mC A\arrow[r,"\iota"] &{\D[\mC C,\mC V]}
   \arrow[l, shift left, bend left, "\iota_R"] 
   \arrow[l, shift right, bend right,swap ,"\iota_L"]
   \arrow[r,"\rho"]			
& {\D[\mC A,\mC V]}\arrow[l, shift left, bend left, "\rho_R"] 
   \arrow[l, shift right, bend right,swap,"\rho_L"] 
\end{diagram*}
where $\mC A \subset \mC C$, $\mC E_\mC A := \{Y\in \mC D[\mC C, \mC V] \mid H_n(Y(a)) = 0
\text{ for all } a\in \mC A\text{ and }n \in \bB Z\}$, $\iota$ is the inclusion, $\rho$ is the
restriction. The functor $\rho_R$ takes $Y$ to $r_R(kY)$, where $kY$ is a  \K-injective resolution
of $Y$ in $\D[\mC A,\mC V]$ and $r_R$ is the functor of Theorem~\ref{Egg} which applies to $kY$ degreewise.
The functor $\rho_L:=\tau_L\circ F^{-1}$, where $\tau_L:\mC T_{\mC A}\to\D[\mC C,\mC V]$ is the inclusion and 
$F:\mC T_{\mC A}\to\D[\mC A,\mC V]$ is the equivalence of Lemma~\ref{F}.
\end{theorem}

\begin{proof}
Using Lemma~\ref{perp}, we have the triangulated functors
$\iota_L, \iota_R$ as defined in Lemma~\ref{comp}. It is clear that $\mC E_\mC
A$ is the kernel of $\rho.$ The functor $\rho_R:\D[\mC A, \mC V]\to \D[\mC C, \mC V]$
is right adjoint to $\rho$ as 
$$\D[\mC C,
\mC V](X, r_RkY)\cong \K[\mC C,\mC V](X, r_RkY)$$
by the fact that  $r_R$ preserves \K-injectivity (see proof of Lemma
\ref{rest}) and $$\K[\mC C,
\mC V](X, r_RkY)\cong \K[\mC A,\mC V](rX, kY)\cong \D[\mC A,\mC V](\rho X, kY)\cong
\D[\mC A,\mC V](\rho X, Y).$$  It is
well known that the adjoint to a triangulated functor is also triangulated, see Lemma~\ref{triadj}, 
thus $\rho_R$ is a triangulated functor. Further,
fully faithfulness follows as $$\D[\mC C,
\mC V](r_RkX, r_RkY)\cong \D[\mC A,\mC V](rr_RkX, Y)$$ by adjunction and 
   $$\D[\mC A,\mC V](rr_RkX, Y)\cong \D[\mC A,\mC V](kX,Y)\cong \D[\mC A,\mC V](X,Y)$$ 
by the fact that $rr_R\To \id_{[\mC A,\mC V]}$ is an
isomorphism of functors. 

Now we claim that
$\rho_L$ is left adjoint to $\rho.$ By construction, $F=\rho\circ\tau_L$. 
By Lemma~\ref{comp} we have 
$$\D[\mC C,\mC
V](\tau_L\circ(\rho\circ\tau_L)^{-1}X,Y)\cong \D[\mC A,\mC V](X,
(\rho\circ\tau_L)\circ \tau Y),$$ 
natural in $X\in \D[\mC A,\mC V]$ and $Y\in
\D[\mC C,\mC V].$ 
By Proposition~\ref{adjtri} there is a distinguished triangle 
$\iota\iota_L Y[-1]\to \tau_L\tau Y\to  Y\to \iota\iota_L Y$.
One has $\iota_L Y\in \mC E_\mC A$. Applying $\rho$ to
the triangle, one gets a triangle $0\to\rho \tau_L\tau Y\to \rho Y\to 0.$ 
Thus $\rho\circ\tau_L\circ\tau
Y\cong \rho Y.$ We have the desired adjunction. Since $\rho_L$ is the composition of fully
faithful functors, it is itself fully faithful. This completes the construction of the recollement.
\end{proof}

The following result is an extension of the recollement of Theorem~\ref{Egg2} to derived categories.

\begin{theorem}\label{lambda}
Under the assumptions of Theorem~\ref{cond} there exists
a recollement of triangulated categories 
\begin{diagram*}[column sep=huge]
\mC E_\mC A\arrow[r,"\iota"] &{\D[\mC C,\mC V]}
   \bendR{\iota_R} 
     \arrow[r,"\lambda"]
	 \bendL{\iota_L}		
& {\D([\mC C,\mC V]/\mC S_\mC A)}
   \bendR{\lambda_R} 
   \bendL{\lambda_L}
\end{diagram*}
where $\mC A \subset \mC C$, $\iota_L,\iota,\iota_R$ are those of Theorem~\ref{TA}.
The functor $\lambda$ applies the $\mC S_\mC A$-localization exact functor to any complex of $\D[\mC C,\mC V]$.
The functor $\lambda_L:=\rho_L\circ\varkappa^{-1}$, where $\rho_L$ is from Theorem~\ref{TA}
and $\varkappa=\ell\circ r_L:[\mC A,\mC V]\xrightarrow{\cong}[\mC C,\mC V]/\mC S_\mC A$
is the equivalence~\eqref{kappa}. The functor $\lambda_R$ takes $Y$ to $\ell_R(kY)$, where 
$kY$ is a \K-injective resolution of $Y$ in $\D([\mC C,\mC V]/\mC S_\mC A)$ and $\ell_R$
is the inclusion functor $[\mC C,\mC V]/\mC S_\mC A\to[\mC C,\mC V]$ of Theorem~\ref{Egg2}
which applies to $kY$ degreewise.
\end{theorem}

\begin{proof} 
We have that $\mC E_\mC A$ is the kernel of
$\lambda$ and $\lambda_L$ is plainly fully faithful. Further consider the natural isomorphisms 
\begin{align*}
\D[\mC C,\mC V](\lambda_LX,Y)&=\D[\mC C,\mC V](\rho_L\circ(\ell\circ
  r_L)^{-1}X,Y)\\&\cong\D([\mC C,\mC V]/\mC S_\mC A)(X,(\ell\circ
  r_L)\circ\rho Y)\\&\cong\D([\mC C,\mC V]/\mC S_\mC A)(X,(\ell\circ r_L
  \circ r)Y).
\end{align*}
By Proposition~\ref{adjseq} and Theorem~\ref{Egg} there is an exact sequence in $\Ch[\mC C,\mC V]$
   $$0\to i(A)\to r_L\circ rY\to Y \to i\circ i_L Y\to 0$$ 
for some $A\in\Ch\mC S_\mC A\subset\mC E_{\mC A}$.
Applying the $\mC S_\mC A$-localization functor $\lambda$ to this exact sequence,
we get that the morphism of functors $\lambda\circ r_L \circ r\to \lambda$ is pointwise an isomorphism
in $\Ch([\mC C,\mC V]/\mC S_{\mC A})$. Hence we can conclude that 
  $$\D([\mC C,\mC V]/\mC S_\mC A)(X,\ell\circ
  r_L\circ r Y)\cong\D([\mC C,\mC V]/\mC S_\mC A)(X,\lambda Y)$$
and $\lambda_L\dashv\lambda.$ We also have $\lambda\dashv\lambda_R$ by 
observing that $\ell_R$ preserves \K-injective resolutions and applying the same arguments 
when proving that $\rho\dashv\rho_R$ in Theorem~\ref{TA}. Now $\lambda_R$
is fully faithful by Lemma~\ref{mm}.
\end{proof}

\section{Recollements for triangulated categories and Serre localization}

After constructing a recollement of $\D[\mC C,\mC V]$ in Theorem~\ref{lambda}, it is important to extend it further
to a recollement of $\D\bigl([\mC C,\mC V]/\mC Q\bigr)$, where $\mC Q$ is a Serre localizing subcategory 
of $[\mC C,\mC V]$. We are motivated to investigate the categorical aspects of Voevodsky's
triangulated categories of motives $\mathbf{DM}^{eff}_{\mC A}(k)$, where $\mC A$ is a reasonable category of 
correspondences on smooth algebraic varieties (see Section~\ref{voev} below). However, such an
extension of the recollement of Theorem~\ref{lambda} to $\D\bigl([\mC C,\mC V]/\mC Q\bigr)$ is hardly possible
for any localizing subcategory $\mC Q$ of $[\mC C,\mC V]$, because certain adjoint functors 
do not seem to be constructible for general $\mC Q$. 

Therefore to make the desired extension of the recollement to $\D\bigl([\mC C,\mC V]/\mC Q\bigr)$ possible,
we need to find the right conditions on the localizing subcategories $\mC S_{\mC A}$ and $\mC Q$
of $[\mC C,\mC V]$. These conditions originate in the fundamental Voevodsky theorem~\cite{Voe1}, which says that
the Nisnevich sheaf $F_{nis}$ associated with a homotopy invariant presheaf with transfers $F$ is 
a homotopy invariant sheaf with transfers
and that it is strictly homotopy invariant whenever the base field is perfect. Translating this theorem into the language
of Serre and Bousfield localization theory in Grothendieck categories and their derived categories, we make the following

\begin{definition}\label{vprop}
Consider a Grothendieck category $\mC D,$ and any two Serre
localizing subcategories $\mC Q,\mC S\subset\mC D.$ We say that $\mC S$ \emph{satisfies
the Voevodsky property with respect to $\mC Q$\/} or just the $V$-\emph{property\/} (``$V$" for Voevodsky) if
the full subcategory in $\mC D$ of $\mC Q$-local objects $\mC S^{\mC Q}=\{S_{\mC Q}\mid S\in\mC S\}$ is a subcategory of $\mC S$.
In other words, the $\mC Q$-localization $S_{\mC Q}$ of any object $S\in\mC S$ is in $\mC S$.
\end{definition}

\begin{example}\label{pwt}
Let $\mC D$ be the Grothendieck category $PwT/k$ of Nisnevich pre\-sheaves with transfers over a field $k$,
$\mC S$ be the Serre localizing subcategory of homotopy invariant presheaves with transfers, and $\mC Q$
be the localizing subcategory of Nisnevich locally trivial presheaves with transfers. Then the Voevodsky
theorem~\cite{Voe1} implies that $\mC S$ satisfies the Voevodsky property with respect to $\mC Q$, because
the $\mC Q$-localization functor is nothing but the Nisnevich sheafification what follows from the following lemma.
\end{example}

\begin{lemma}\label{nissh}
The quotient category $(PwT/k)/\mC Q$ of $\mC Q$-local presheaves with transfers equals the category
$SwT/k$ of Nisnevich sheaves with transfers.
\end{lemma}

\begin{proof}
Given $F\in PwT/k$ its Nisnevich sheaf $F_{nis}$ is in $PwT/k$ and the morphism
$F\to F_{nis}$ is a morphism in $PwT/k$ by~\cite{Voe1} (see~\cite[1.2]{SV1} as well).
Obviously, $F_{nis}$ is $\mC Q$-torsionfree. Let
   $$F_{nis}\hookrightarrow G\twoheadrightarrow X$$
be a short exact sequence in $PwT/k$ with $X\in\mC Q$. Since $X_{nis}=0$, it follows that
the composite map $F_{nis}\hookrightarrow G\to G_{nis}$ is an isomorphism in $PwT/k$,
hence the short exact sequence splits. We see that $\Ext^1_{PwT/k}(X,F_{nis})=0$,
and so $F_{nis}$ is $\mC Q$-local. Thus every Nisnevich sheaf with transfers is $\mC Q$-local.

Conversely, suppose $F\in PwT$ is $\mC Q$-local. Consider a long exact sequence
in the category $\mathbb ZPre(Sm/k)$ of ordinary presheaves of Abelian groups 
   $$A\hookrightarrow F\to F_{nis}\twoheadrightarrow X,$$
where $A,X$ are Nisnevich locally trivial. Since $F\to F_{nis}$ is a map of $PwT/k$,
it follows that $A,X\!\!\in\! PwT$ (limits/colimits in $PwT\!/k$ are computed in $\mathbb ZPre(Sm/k)$).
Therefore $A,X\in\mC Q$. But $F$ is $\mC Q$-local, and so $A=0$. We get a
short exact sequence $F\hookrightarrow F_{nis}\twoheadrightarrow X$ in $PwT/k$.
Since $F\in PwT$ is $\mC Q$-local and $X\in\mC Q$, this short exact sequence splits.
We have $F_{nis}=F\oplus X$. But $F_{nis}$ is $\mC Q$-torsionfree, so $X=0$ and $F=F_{nis}$.
Thus every $\mC Q$-local presheaf with transfers is a Nisnevich sheaf with transfers.
\end{proof}

Since we work with two localizing subcategories, we need the notion of their join.

\begin{definition}
Consider a Grothendieck category $\mC D,$ and any two Serre
localizing subcategories $\mC Q,\mC S\subset\mC D.$ We define the \emph{join\/} $\mC J:=\sqrt{(\mC Q\cup
\mC S )}$ of $\mC Q$ and $\mC S$ as the smallest localizing subcategory containing them.
\end{definition}

Below we shall need the following lemma proven in~\cite[Lemma 5]{Gar09}.

\begin{lemma}\label{L5}
Given two localizing subcategories $\mC Q,\mC S$ of a
Grothendieck category $\mC D$, the full subcategory $\mC S^{\mC Q}$ of objects of the
form $X_{\mC Q}$ with $X\in\mC S$ is closed under direct sums, subobjects, and
quotient objects in $\mC D/\mC Q.$ Moreover, let $\mC J$ be the unique localizing subcategory
of $\mC D$ containing $\mC Q$ such that the quotient Grothendieck category 
$\mC D/\mC J=(\mC D/\mC Q)/\sqrt{\mC S^{\mC Q}}$, where
$\sqrt{\mC S^{\mC Q}}$ is the smallest localizing category in $\mC D/\mC Q$ containing $\mC S^{\mC Q}$. Then
$\mC J=\sqrt{(\mC Q\cup\mC S)}$, the join of $\mC Q,\mC S$ in $\mC D$.
\end{lemma}

\begin{lemma}\label{SAQ}
Let $\mC S,\mC Q$ be two localizing subcategories of a Grothendieck category $\mC D$. 
If $\mC S$ satisfies the Voevodsky property with respect to $\mC Q$, then
the subcategory $\mC S^\mC Q$ of Lemma~\ref{L5} is localizing in 
the quotient Grothendieck category $\mC D/\mC Q.$
\end{lemma}

\begin{proof}
By definition, $\mC S^\mC Q\subset \sqrt{\mC S^\mC Q}.$ Further,
$\mC S^\mC Q$ is closed under direct sums, subobjects and quotients 
in $\mC D/\mC Q$ by Lemma~\ref{L5}. Consider an exact sequence 
$X\hookrightarrow Y \overset{g}{\twoheadrightarrow}Z$ with $X,Z\in \mC S^\mC Q.$ 
By assumption, $\mC S^\mC
Q\subset \mC S$ and so $X,Z\in \mC S$. Moreover,  the image of $g$ in $\mC D$ 
also belongs to $ \mC S$ as $\Im(g)\subset Z$. Thus there is a
short exact sequence $X\hookrightarrow Y \overset{g}{\twoheadrightarrow }\Im(g)$
in $\mC D$ with $X,\Im(g)\in \mC S$. This implies that $Y\in\mC
S$, hence  $Y=Y_\mC Q\in\mC S^\mC Q$ and we have that $\mC S^\mC Q$ 
is closed under extensions. We conclude that $\mC S^\mC Q$ is localizing in $\mC D/\mC Q.$
\end{proof}

\begin{definition}\label{vsprop}
Under the assumptions of the preceding lemma denote by
$\mC E:=\{Y\in\D(\mC D)\mid H^*(Y)\in\mC S\}$ and 
$\mC E^{\mC Q}:=\{Z\in\D(\mC D/\mC Q)\mid H^*_{\mC D/\mC Q}(Z)\in\mC S^{\mC Q}\}$,
where $H^*_{\mC D/\mC Q}$ refers to cohomology computed in the quotient category
${\mC D/\mC Q}$ (note that $\mC S^{\mC Q}$ is localizing in $\mC D/\mC Q$ by the preceding lemma). We have
a canonical triangulated functor between derived categories
   $$k:\D(\mC D/\mC Q)\to\D(\mC D)$$
taking a complex $Z\in\D(\mC D/\mC Q)$ to its \K-injective resolution $kY$ in $\D(\mC D/\mC Q)$
and regarding $kY$ as a complex of $\D(\mC D)$. Notice that $k$ preserves \K-injective objects.

Consider a Grothendieck category $\mC D$ and any two Serre
localizing subcategories $\mC Q,\mC S\subset\mC D.$ We say that 
$\mC S$ \emph{satisfies the strict Voevodsky property with respect to $\mC Q$\/} or just the \emph{strict $V$-property\/}
if $\mC S$ satisfies the Voevodsky property with respect to $\mC Q$ and the restriction of
$k:\D(\mC D/\mC Q)\to\D(\mC D)$ to $\mC E^{\mC Q}$ lands in $\mC E$. In other words, the $\mC D$-homology objects
of the \K-injective resolutions of the complexes from $\mC E^{\mC Q}$ belong to $\mC S$.  
\end{definition}

\begin{example}\label{swt}
Under the notation of Example~\ref{pwt} $\mC D/\mC Q$ is the derived category $\D(SwT/k)$
of the Grothendieck category $SwT/k$ of Nisnevich sheaves with transfers by Lemma~\ref{nissh}. If the base field
$k$ is perfect, the Voevodsky theorem~\cite{Voe1} and~\cite[6.2.7]{Mor} imply that $\mC S$ satisfies the 
strict Voevodsky property with respect to $\mC Q$.
\end{example}
 
\begin{proposition}\label{join}
Let $\mC J$ be the join of two localizing subcategories $\mC Q,\mC S$ of a Grothendieck category $\mC D$.
If we consider quotient categories as full
subcategories of local objects then the following relation is true: $\mC D/\mC J=(\mC D/\mC S)\cap(\mC D/\mC Q).$ In
other words, $\mC D/\mC J$ is the full subcategory of those objects that are both $\mC Q$- and $\mC S$-local.
\end{proposition}
 
\begin{proof}
Since $\mC D/\mC J=(\mC D/\mC S)/\sqrt{\mC Q^{\mC S}}$ by Lemma~\ref{L5}, this implies that
every object in $\mC D/\mC J$ is $\mC S$-local. In the same manner, every object in 
$\mC D/\mC J$ is $\mC Q$-local and hence $\mC D/\mC J\subset(\mC D/\mC S)\cap(\mC D/\mC Q).$ 
Next, suppose $X\in(\mC D/\mC S)\cap(\mC D/\mC Q)$. It is enough to show that 
$X$ is $\sqrt{\mC Q_{\mC S}}$-local in $\mC D/\mC S.$ 
By Lemma~\ref{L5} $\mC Q_{\mC S}$ is closed under subobjects, quotients and direct sums.
By~\cite[Corollary 3]{Gar09} $X$ is $\sqrt{\mC Q_{\mC S}}$-local in $\mC D/\mC S$
if and only if $\Hom_{\mC D/\mC S}(W_{\mC S},X)=0=\Ext^1_{\mC D/\mC S}(W_{\mC S},X)$ for all
$W_{\mC S}\in\mC Q_{\mC S},$ where $W_{\mC S}\in\mC Q_{\mC S}$ is the $\mC S$-localization of some $W\in\mC Q.$
We have $\Hom_{\mC D/\mC S}(W_{\mC S},X)=\Hom_\mC D(W,X)=0,$ as $X$ is $\mC Q$-local. Now, for
 $\Ext^1_{\mC D/\mC S}(W_{\mC S},X)=0,$ we show that every short exact sequence splits. 
 
Suppose we have
a short exact sequence ${X\hookrightarrow M \twoheadrightarrow
W_{\mC S}}$ in $\mC D/\mC S.$ If we consider the canonical morphism $\lambda$ associated with
$\mC S$-localization, then there is the following diagram in \mC D
\begin{diagram*}
X \arrow[r,hookrightarrow,"i"]\arrow[d,"="]& N
\arrow[r,"\alpha"]\arrow[d,"\lambda_N"] & W\arrow[d, "\lambda_W"]\\
X  \arrow[r,hookrightarrow,"i'"]& M \arrow[r] & W_{\mC S}.
\end{diagram*}
As $\mC Q$ is localizing, if we take the image of $\alpha$ in $W$ then $\Im
\alpha\in\mC Q$. Since $X$ is $\mC Q$-local then $\Ext^1_{\mC D}(\Im\alpha,X)=0$.
Thus the short exact sequence ${X\hookrightarrow N
\twoheadrightarrow\Im\alpha}$ splits in $\mC D$. Therefore there
exists some map $\beta:N\to X$ with $\beta\circ i=1_X$, and since $X$ is $\mC
S$-local and $M=N_{\mC S}$, we have that $\beta$ factors uniquely through $\lambda_N.$ Denote this factor
$\beta':M\to X$ with $\beta'\circ \lambda_N=\beta.$ We see that $\beta'$ splits
the bottom short exact sequence in $\mC D/\mC S$ as $1_X=\beta i=(\beta' \lambda_N)i=\beta' i'.$
Hence $X$ is $\sqrt{\mC Q_{\mC S}}$-local. It follows that $X$ is $\mC J$-local and $(\mC
D/\mC S)\cap(\mC D/\mC Q)\subset \mC D/\mC J.$ We conclude that $\mC
 D/\mC J=(\mC D/\mC S)\cap(\mC D/\mC Q).$
\end{proof}

We are now in a position to formulate the main result of this section.

\begin{theorem}\label{bigthm}
Under the assumptions of Theorem~\ref{cond} suppose $\mC A\subset\mC C$ and $\mC Q\subset [\mC C,\mC V]$ 
is a localizing subcategory such that: 
\begin{itemize}
  \item $\D\big([\mC C, \mC V]/\mC Q \big)$ is
  compactly generated and the functor $\D[\mC C,\mC V]\to \D\big([\mC C,\mC V]/\mC
  Q\big)$ induced by the exact $\mC Q$-localization functor $(\cdot)_{\mC Q}:[\mC C,\mC V]\to[\mC C,\mC V]/\mC
  Q$ respects compact objects;
  \item the localizing subcategory $\mC S_\mC A=\bigl\{Y\in[\mC C,\mC V]\mid Y(a)=0\textrm{ for all $a\in\mC A$}\bigr\}$ 
  satisfies the strict Voevodsky property with respect to $\mC Q$.
\end{itemize}
Then there exists a recollement of triangulated categories \begin{diagram*}[column sep=huge]
\mC E_\mC A^\mC Q\arrow[r,"\iota^\mC Q"] &{\D\bigl([\mC C,\mC V]/\mC Q\bigr)}
   \bendR{\iota_R^\mC Q} 
   \bendL{\iota_L^\mC Q}
   \arrow[r,"\lambda^\mC Q"]			
& {\D\bigl([\mC C,\mC V]/\mC J_\mC A\bigr),}
   \bendR{\lambda_R^\mC Q} 
   \bendL{\lambda_L^\mC Q}
\end{diagram*}
where $\mC E_\mC A^\mC Q$ is as in Definition~\ref{vsprop},
$\mC J_\mC A$ is the join of $\mC Q$ and $\mC S_\mC A$ in $[\mC C,\mC V]$, the
functor $\lambda^{\mC Q}$ is induced by the $\mC S_\mC A^\mC Q$-localization functor 
$[\mC C,\mC V]/\mC Q\to[\mC C,\mC V]/\mC J_\mC A$ associated with the localizing subcategory
$\mC S_\mC A^\mC Q=\{X_{\mC Q}\mid X\in\mC S_{\mC A}\}$ of Lemmas~\ref{L5}-\ref{SAQ}, $\iota^{\mC Q}$ is inclusion,
and $\lambda_R^\mC Q$ is induced by the \K-injective resolution functor.
\end{theorem}

\begin{remark}\label{remdm}
In our major application (see Section~\ref{voev}) the middle category will be the derived 
category of Nisnevich sheaves with reasonable transfers $\mC B$ on smooth algebraic varieties
$Sm/k$, the left category is Voevodsky's~\cite{Voe2} triangulated category of 
motives $\mathbf{DM}^{eff}_{\mC B}(k)$. The celebrated Voevodsky theorem~\cite[3.2.6]{Voe2} computes
the functor $\iota_L^{\mC Q}$ for this example as the Suslin complex $C_*$. Theorem~\ref{bigthm}
is a kind of the categorical framework for the triangulated category of motives $\mathbf{DM}^{eff}_{\mC B}(k)$.
\end{remark}

We postpone the proof of Theorem~\ref{bigthm}. First we establish a number of facts.
Note that $[\mC C,\mC V]/\mC J_\mC A\cong \bigl([\mC C,\mC V]/\mC Q\bigr)/\mC S_\mC A^\mC Q$ by 
Lemma~\ref{L5}. By Proposition~\ref{join} $[\mC C,\mC V]/\mC J_\mC A$ consists of those
objects of $[\mC C,\mC V]$ which are both $\mC Q$- and $\mC S_{\mC A}$-local. 
Recall from Theorem~\ref{cond} that $\D[\mC C,\mC V]$ is compactly generated with
compact generators $\bigl\{\mC V_\mC C(c,-)\oslash Q_j\bigr\}$.

\begin{lemma}\label{TAQ}
Under the assumptions of Theorem~\ref{bigthm} denote by $\mC T_\mC A^\mC Q$ 
the full subcategory of $\D\bigl([\mC C,\mC V]/\mC Q\bigr)$
generated by the compact objects $\Big\{\big(\mC V_\mC C(a,-)\oslash
Q_j\big)_\mC Q\mid a\in\mC A\Big\}$. Then there is a recollement of triangulated categories
 \begin{diagram*}[column sep=huge]
\mC E_\mC A^\mC Q\arrow[r,"\iota^\mC Q"] &{\D([\mC C,\mC
V]/\mC Q)}
   \bendR{\iota_R^\mC Q} 
   \bendL{\iota_L^\mC Q}
   \arrow[r,"\tau^\mC Q"]			
& {\mC T_\mC A^\mC Q}
	\bendR{\tau_R^\mC Q} 
    \bendL{\tau_L^\mC Q}
\end{diagram*}
where $\iota^\mC Q,\tau_L^\mC Q$ are inclusions.
\end{lemma}

\begin{proof} Let us show that $\mC
E_\mC A^\mC Q=(\mC T_\mC A^\mC Q)^\perp$.
Suppose $X\in\mC E_\mC A^\mC Q$, then
\begin{align*}
\D\bigl([\mC C,\mC V]/\mC Q\bigr)\Bigl(\bigl(\mC V_\mC C(a,-)\oslash Q_j\bigr)_\mC
Q,X\Bigr)&=\K\bigl([\mC C,\mC V]/\mC Q\bigr)\Bigl(\bigl(\mC V_\mC C(a,-)\oslash
Q_j\bigr)_\mC Q,kX\Bigr)\\
&=\K[\mC C,\mC V]\Bigl(\mC V_\mC C(a,-)\oslash
Q_j,kX\Bigr),
\end{align*} 
where $kX$ is the \K-injective resolution of $X$ in $\D\bigl([\mC C,\mC V]/\mC Q\bigr)$.
Since $\mC S_{\mC A}$ satisfies the strict $V$-property by assumption, it follows from
Lemma~\ref{perp} that $kX\in\mC E_\mC A=\mC T_\mC A^\perp$ and then $\K[\mC C,\mC
V]\bigl(\mC V_\mC C(a,-)\oslash Q_j,kX\bigr)=0$. Therefore $X\in (\mC T_\mC A^\mC Q)^\perp$, 
because by a theorem of Neeman~\cite[2.1]{Nee96} 
$\mC T_\mC A^{\mC Q}$ is the smallest triangulated full subcategory closed under direct sums 
containing $\bigl(\mC V_\mC C(a,-)\oslash Q_j\bigr)_{\mC Q}$-s. Next, given $Y\in(\mC T_\mC
A^\mC Q)^\perp$ then
\begin{align*}
0&=\D\big([\mC C,\mC V]/\mC Q\big)\Big(\big(\mC V_\mC C(a,-)\oslash
Q_j\big)_\mC Q,kY\Big)\\
&=\K\big([\mC C,\mC V]/\mC Q\big)\Big(\big(\mC V_\mC C(a,-)\oslash
Q_j\big)_\mC Q,kY\Big)\\
&=\K[\mC C,\mC V]\Big(\mC V_\mC C(a,-)\oslash Q_j,kY\Big)\\
&=\D[\mC C,\mC V]\Big(\mC V_\mC C(a,-)\oslash Q_j,kY\Big)\\
&\cong\D(\mC V)\Big(
Q_j,kY(a)\Big).
\end{align*}
We use here the isomorphism~\eqref{tenshom} and the fact that $kY$ is \K-injective in $\D[\mC C,\mC V]$.
This implies that $kY(a)$ is acyclic, because $Q_j$-s are compact generators of $\D(\mC V)$. 
Thus $kY$ has cohomology belonging to $\mC S_\mC A$ and therefore its $\mC Q$-localized cohomology belongs to $\mC S_\mC
A^\mC Q.$ Since $H^*_{[\mC C,\mC V]/\mC Q}(Y)=H^*_{[\mC C,\mC V]/\mC Q}(kY)\in\mC S_\mC A^\mC Q$ 
we have that $Y\in\mC E_\mC A^\mC Q.$ We conclude
that $\mC E_\mC A^\mC Q=\bigl(\mC T_\mC A ^\mC Q\bigr)^\perp\!.$ 
Since $\mC T_\mC A ^\mC Q$ is compactly generated, our statement now follows from~\cite[5.6.1]{Krause2}.
\end{proof} 

\begin{lemma}\label{kkk}
Under the assumptions of Theorem~\ref{bigthm}
let $X\in\D\bigl([\mC C,\mC V]/\mC Q\bigr)$ and $Y=\lambda^\mC Q(X)$. Then the natural morphism $f_a:kX(a)\to
kY(a)$ induced by $X\to Y$ is a quasi-isomorphism in $\Ch(\mC V)$ for all $a\in\mC A$, where $kX,kY$ are \K-injective resolutions
of $X,Y$ in $\D\bigl([\mC C,\mC V]/\mC Q\bigr)$. Furthermore, if $k'Y$ is a \K-injective resolution of
$Y$ in $\D([\mC C,\mC V]/\mC J_{\mC A})$, then the induced morphisms $g_a:kY(a)\to
k'Y(a)$ and $g_af_a:kX(a)\to k'Y(a)$ are quasi-isomorphisms in $\Ch(\mC V)$ for all $a\in\mC A$.
\end{lemma}

\begin{proof}
Consider a triangle $Z\to X\overset{\phi}{\longrightarrow} Y\to Z[1]$
in $\D\bigl([\mC C,\mC V]/\mC Q\bigr)$.
We have then $Z\in \mC E^\mC Q_\mC A$, as $\Ker\lambda^\mC Q=\mC E_\mC A^\mC Q$
and $\lambda^\mC Q(\phi)$ is an isomorphism in $\Ch\bigl([\mC C,\mC V]/\mC J_{\mC A}\bigr)$.
Apply the triangulated functor of Definition~\ref{vsprop} $k:\D\bigl([\mC C,\mC V]/\mC Q\bigr)\to\D[\mC C,\mC V]$ 
and consider the triangle $kZ\to kX\overset{f}
\to kY\to kZ[1]$ in $\D[\mC C,\mC V]$, where $f=k(\phi)$. We have $kZ\in \mC E_\mC A$ in
$\D[\mC C,\mC V]$ by an assumption of Theorem~\ref{bigthm}. Hence
$kZ(a)\cong 0$ and $f_a:kX(a)\xrightarrow\cong kY(a)$ in $\D(\mC V)$.
We use here the triangulated evaluation functor $\D[\mC C,\mC V]\to\D(\mC V)$
induced by the exact evaluation functor $B\in[\mC C,\mC V]\mapsto B(a)\in\mC V$.

Now consider a triangle $W\to Y \overset{\gamma}{\longrightarrow} k'Y\to W[1]$
in $\D\bigl([\mC C,\mC V]/\mC Q\bigr)$. Then $\lambda^\mC Q(\gamma)$ is a quasi-isomorphism in 
$\Ch\bigl([\mC C,\mC V]/\mC J_{\mC A}\bigr)$, and hence $W\in \mC E^\mC Q_\mC A$. Note that
$k'Y$ is \K-injective in $\D\bigl([\mC C,\mC V]/\mC Q\bigr)$.
As above the induced 
morphism $g:kY\to k'Y$ is such that $g_a:kY(a)\to k'Y(a)$ is a quasi-isomorphisms in 
$\Ch(\mC V)$ for all $a\in\mC A$, and hence so is $g_af_a:kX(a)\to k'Y(a)$.
\end{proof}

\begin{corollary}\label{FQ1} 
The functor $\lambda^\mC Q$ induces an isomorphism
$$\D\bigl([\mC C,\mC V]/\mC Q\bigr)\Bigl(\bigl(\mC V_\mC C(a,-)\oslash P\bigr)_\mC Q,X\Bigr)\xrightarrow\cong\mC \D\bigl([\mC
C,\mC V]/\mC J_\mC A\bigr)\Bigl(\bigl(\mC V_\mC C(a,-)\oslash P\bigr)_{\mC J_{\mC A}},Y\Bigr)$$
for all $a\in\mC A$, $P\in\{Q_j\}_J$, $X\in\D\bigl([\mC C,\mC V]/\mC Q\bigr)$ and $Y=\lambda^\mC Q(X)$.
\end{corollary}

\begin{proof}
This follows from Lemma~\ref{kkk} isomorphism~\eqref{tenshom} 
and the commutativity of the following diagram, in which all vertical arrows are isomorphisms:
\begin{diagram*}[column sep=small]
\D\bigl([\mC C,\mC V]/\mC Q\bigr)\Bigl(\bigl(\mC V_\mC C(a,-)\oslash P\bigr)_\mC Q,X\Bigr) \arrow[r,
"\lambda^\mC Q_*"]\arrow[d,"\cong"]& \mC \D\bigl([\mC C,\mC
V]/\mC J_\mC A\bigr)\Bigl(\bigl(\mC V_\mC C(a,-)\oslash P\bigr)_{\mC J_{\mC A}},Y\Bigr)\arrow[d,"\cong"]\\
\D\bigl([\mC C,\mC V]/\mC Q\bigr)\Bigl(\bigl(\mC V_\mC C(a,-)\oslash P\bigr)_\mC Q,kX\Bigr)
\arrow[d,"\cong"]& \D\bigl([\mC C,\mC V]/\mC J_\mC A\bigr)\Bigl(\bigl(\mC V_\mC C(a,-)\oslash P\bigr)_{\mC J_{\mC A}},k'Y\Bigr)\arrow[d,"\cong"]\\
\D[\mC C,\mC V]\bigl(\mC V_\mC C(a,-)\oslash P,kX\bigr) \arrow[d,"\cong"]& \D[\mC
C,\mC V]\bigl(\mC V_\mC C(a,-)\oslash P,k'Y\bigr)\arrow[d,"\cong"]\\
\D(\mC V)\bigl(P,kX(a)\bigr) \arrow[r,"\cong"]& \D(\mC V)\bigl(P,k'Y(a)\bigr)
\end{diagram*}
where $kX$ is a \K-injective resolution of $X$ in $\D\bigl([\mC C,\mC V]/\mC Q\bigr)$
and $k'Y$ is a \K-injective resolution of $Y$ in $\D\bigl([\mC C,\mC V]/\mC J_{\mC A}\bigr)$.
\end{proof}

\begin{corollary} 
The triangulated functor 
$F^\mC Q:=\lambda^\mC Q\circ \tau_L^\mC Q:\mC T_\mC A^\mC
Q\to\D\bigl([\mC C,\mC V]/\mC J_\mC A\bigr)$ is fully faithful.
\end{corollary}

\begin{proof}
It follows from Corollary~\ref{FQ1} that the functor $F^\mC Q$ is fully faithful
on pairs $Z,X\in\mC T_\mC A^\mC Q$  with $Z$ compact. 
Moreover, $F^\mC Q$ preserves direct sums, because direct sums in the
derived category of a Grothendieck category are formed degreewise~\cite[19.13.4]{Stack} 
and $\lambda^\mC Q$ preserves degreewise direct sums.
Let $S\subset \mC T_\mC A^\mC Q$ be the
full triangulated subcategory of those objects $Z$ such that 
$(Z,X)\to(F^\mC Q Z,F^\mC Q X)$ is an isomorphism for a fixed $X\in\mC T_\mC A^\mC Q$. 
Then $S\supset(\mC T_\mC A^\mC Q)^c$ and $S$ is closed under direct sums.
Therefore $S$ is localizing. We
conclude from a theorem of Neeman~\cite[2.1]{Nee96} that $S=\mC T_\mC A^\mC Q$.
\end{proof} 

\begin{lemma}\label{equiv}
The triangulated functor $F^\mC Q:\mC T_\mC A^\mC Q\to\D\bigl([\mC C,\mC V]/\mC J_\mC A\bigr)$ is an equivalence.
\end{lemma}

\begin{proof}
As $F^\mC Q$ is fully faithful by the preceding corollary, it remains to show that it is essentially
surjective. Take any $Y\in \D\bigl([\mC C,\mC V]/\mC J_\mC A\bigr)$ and consider it as an object of
$\D\bigl([\mC C,\mC V]/\mC Q\bigr)$. Then $\lambda^\mC
Q(Y)=Y.$ If we take $\tau^\mC Q (Y)\in\mC T_\mC A^\mC Q$ (see Lemma~\ref{TAQ}), then $F^\mC Q(\tau^\mC QY)=\lambda^\mC
Q\tau_L^\mC Q\tau^\mC Q Y$. Using Lemma~\ref{TAQ} and Proposition~\ref{adjtri}, there exists a
triangle $$\tau^\mC Q_L\tau^\mC Q
Y\to Y \to \iota^\mC Q \iota^\mC Q_L Y\to (\tau^\mC Q_L\tau^\mC Q Y)[1].$$
Applying the triangulated functor $\lambda^\mC Q$ to it, it follows that
$\lambda^\mC Q\tau_L^\mC Q\tau^\mC Q Y\cong \lambda^\mC QY=Y$. 
We have $F^\mC Q(\tau^\mC QY)\cong Y$ and thus $F^\mC Q$ is
essentially surjective. We conclude that $F^\mC Q$ is an equivalence.
\end{proof} 

\begin{proof}[Proof of Theorem~\ref{bigthm}]
We define the functors $\iota^\mC Q_L,\iota^\mC Q,\iota^\mC Q_R$ to be those coming from
Lemma~\ref{TAQ}, and $\Ker\lambda^\mC Q=\mC E_{\mC A}^{\mC Q}$ is immediate.
We set $\lambda^\mC Q_L:=\tau_L^\mC Q\circ (F^\mC Q)^{-1}$. Then $\lambda^\mC Q_L$ 
is fully faithful as it is the composition of fully faithful functors. We have natural isomorphisms
\begin{align*}
\D\bigl([\mC C,\mC V]/\mC Q\bigr)(\lambda^\mC Q_LX,Y)&=\D\bigl([\mC C,\mC V]/\mC Q\bigr)
\bigl(\tau^\mC Q_L\circ(F^\mC Q)^{-1}X,Y\bigr)\\&\cong\D\bigl([\mC C,\mC V]/\mC J_\mC
A\bigr)(X,F^\mC Q\circ\tau^\mC Q Y)\\&\cong\D\bigl([\mC C,\mC V]/\mC J_\mC
A\bigr)(X,\lambda^\mC Q\circ\tau_L^\mC Q\circ\tau^\mC QY)\\
&\cong\D\bigl([\mC C,\mC V]/\mC J_\mC
A\bigr)(X,\lambda^\mC QY).
\end{align*}
We use here the fact that $\lambda^\mC Q\tau_L^\mC Q\tau^\mC Q Y\cong \lambda^\mC QY$
(see the proof of Lemma~\ref{equiv}).
Therefore $\lambda^\mC Q_L\dashv\lambda^\mC Q$. The proof that
$\lambda^\mC Q\dashv\lambda_R^\mC Q$ is obvious. Lemma~\ref{mm} now implies that
$\lambda_R^\mC Q$ is fully faithful. Hence the theorem.
\end{proof}

We can summarise Theorems~\ref{lambda}-\ref{bigthm} by drawing the following diagram
\begin{diagram*}[column sep=huge,row sep=huge]
\mC E_\mC A \arrow[r,"\iota" description] \arrow[d, shift
right=1ex,"(\cdot)_\mC
Q" swap]&{\D[\mC C,\mC V]}\arrow[d, shift right=1ex,"(\cdot)_\mC
Q" swap]
   \arrow[l, shift right=1em,swap ,"\iota_L"] 
   \arrow[l,shift left=1em, "\iota_R"]
   \arrow[r,"\lambda" description]			
& {\D\bigl([\mC C,\mC V]/\mC S_\mC A\bigr)}\arrow[d, shift right=1ex,"(\cdot)_\mC
Q'" swap]\arrow[l,swap, shift right=1em,"\lambda_L"] \arrow[l,shift
left=1em,"\lambda_R"] \\
 \mC E_\mC A^\mC Q\arrow[u,"k" swap]\arrow[r,"\iota^\mC Q" description]
 &{\D\bigl([\mC C,\mC V]/\mC Q\bigr)}\arrow[u,"k" swap]
   \arrow[l, shift right=1em,swap, "\iota_L^\mC Q"] 
   \arrow[l,shift left=1em,"\iota_R^\mC Q"]
   \arrow[r,"\lambda^\mC Q" description]			
& {\D\bigl([\mC C,\mC V]/\mC J_\mC A\bigr)}\arrow[l,swap, shift right=1em,
"\lambda_L^\mC Q"]\arrow[u,"k'" swap]
   \arrow[l, shift left=1em,"\lambda_R^\mC Q"]
\end{diagram*}
in which the top and bottom horizontal arrows form recollements, the vertical arrows
$k,k'$ take \K-injective resolutions in the derived categories of the corresponding quotient
Grothendieck categories, the functor $(\cdot)_\mC Q$ is induced by $\mC Q$-localization
and $(\cdot)_\mC Q'$ is induced by the exact $\sqrt{\mC Q_{\mC S_{\mC A}}}$-localization functor 
$[\mC C,\mC V]/\mC S_\mC A\to[\mC C,\mC V]/\mC J_\mC A$ with 
$\mC Q_{\mC S_{\mC A}}=\{X_{\mC S_{\mC A}}\mid X\in\mC Q\}$.

We have the following obvious equivalences of functors:
   $$(\cdot)_\mC Q'\circ\lambda\cong\lambda^{\mC Q}\circ(\cdot)_\mC Q,\quad
       (\cdot)_\mC Q\circ\iota\cong\iota^{\mC Q}\circ(\cdot)_\mC Q,\quad
       \lambda_R\circ k'\cong k\circ\lambda^{\mC Q}_R,\quad
       \iota\circ k\cong k\circ\iota^{\mC Q}.$$
Since $(\cdot)_\mC Q\dashv k$, $\iota^\mC Q_L\dashv\iota^\mC Q$, $\iota_L\dashv\iota$,
it follows that $\iota^\mC Q_L\circ(\cdot)_\mC Q\dashv k\circ\iota^\mC Q$ and
$(\cdot)_\mC Q\circ\iota_L\dashv\iota\circ k$. But $\iota\circ k\cong k\circ\iota^{\mC Q}$, 
hence there is an equivalence of functors
   $$\iota^\mC Q_L\circ(\cdot)_\mC Q\cong(\cdot)_\mC Q\circ\iota_L.$$
We obtain an equivalence of functors 
$\iota^\mC Q_L\circ(\cdot)_\mC Q\circ k\cong(\cdot)_\mC Q\circ\iota_L\circ k$.
Since $(\cdot)_\mC Q\circ k\cong\id$, we see that 
$\iota^\mC Q_L\cong(\cdot)_\mC Q\circ\iota_L\circ k$.

As we have noticed in Remark~\ref{remdm}, the functor $\iota^\mC Q_L$ is of
particular importance in Voevodsky's triangulated categories of motives. It was computed 
by Voevodsky~\cite[3.2.6]{Voe2} as the Suslin complex $C_*$. In order to make this
precise, we convert recollements of Theorems~\ref{lambda}-\ref{bigthm} into the Voevodsky
language of triangulated categories of motives in the next section.

\section{Recollements for triangulated categories of motives}\label{voev}

Throughout this section $k$ is a field and $\mC V=\Ab$. So every $\mC V$-category $\mC C$
is nothing but a preadditive category. We always work here with preadditive categories whose objects are
the $k$-smooth separated schemes of finite type $Sm/k$.  
Such preadditive categories are also called categories of correspondences
in motivic homotopy theory. We fix a category of correspondences 
$\mC C$ in the sense of~\cite{Gar17}. Briefly, $\mC C$ must have a compatible action of $Sm/k$ and
satisfy certain properties with respect to Nisnevich topology. It is a generalisation of the category of
finite correspondences $Cor$ in the sense of Suslin--Voevodsky~\cite{SV1}.

The category of enriched functors $[\mC C^{\op},\mC V]$ is the category of additive contravariant functors from
$\mC C$ to Abelian groups $\Ab$. It is also called the \emph{category of presheaves with $\mC C$-correspondences}.
We shall denote it by $Pre(\mC C)$.  It is a Grothendieck category such that
$\bigl\{\mC C(-,X)\bigr\}_{X\in Sm/k}$ is a family of finitely generated projective generators of
$Pre(\mC C)$. Denote by $\D\bigl(Pre(\mC C)\bigr)$ its derived category of unbounded complexes. We shall
also write $\mathbb Z_{\mC C}(X)$ to denote $\mC C(-,X)$, where $X\in Sm/k$.
The derived category of Abelian groups $\D(\Ab)$ has compact generators given by the shifted
complexes $\mathbb Z[n]$ of $\mathbb Z$, which are \K-projective in $\D(\Ab)$ as well. By Theorem~\ref{cond}(1)
$\D\bigl(Pre(\mC C)\bigr)$ has compact generators given by the shifted
complexes $\mathbb Z_{\mC C}(X)[n]=\mathbb Z_{\mC C}(X)\otimes\mathbb Z[n]$, $X\in Sm/k$,
which are \K-projective in $\D\bigl(Pre(\mC C)\bigr)$ also.

\begin{definition}
Let $\mathbb A^1$ be the affine line and $X\in Sm/k$. The projection morphism $pr_X:X\times\mathbb A^1\to X$
is left inverse to the inclusion $i_0:X\to X\times\mathbb A^1$ identifying $X$ with $X\times 0$ in $X\times\mathbb A^1$.
It induces a split epimorphism of finitely generated projective objects in $Pre(\mC C)$
   $$pr_{X,*}:\mathbb Z_{\mC C}(X\times\mathbb A^1)\to\mathbb Z_{\mC C}(X).$$
We define $\mathbb Z_{\mC C}^{\mathbb A^1}(X):=\Ker(pr_{X,*})$. By definition, 
$\mathbb Z_{\mC C}^{\mathbb A^1}(X)$ is finitely generated projective in $Pre(\mC C)$.

We also define the following full subcategory of $Pre(\mC C)$:
   $$\mC S_{\mathbb A^1}:=\Bigl\{F\in Pre(\mC C)\mid\Hom_{Pre(\mC C)}\bigl(\mathbb Z_{\mC C}^{\mathbb A^1}(X),F\bigr)=0
       \textrm{ for all $X\in Sm/k$}\Bigr\},$$
which is localizing because the $\mathbb Z_{\mC C}^{\mathbb A^1}(X)$-s are projective objects in $Pre(\mC C)$.
The full subcategory of $Pre(\mC C)$ whose objects are the $\mathbb Z_{\mC C}^{\mathbb A^1}(X)$, $X\in Sm/k$,
will be denoted by $\mC C^{\mathbb A^1}$. It will play the same role as the $\mC V$-full subcategory $\mC A$ from
previous sections. The category of contravariant additive functors from $\mC C^{\mathbb A^1}$ to $\Ab$
will be denoted by $Pre(\mC C^{\mathbb A^1})$. It will play the same role as $[\mC A,\mC V]$ from previous sections.
\end{definition}

Recall that a presheaf $F$ of Abelian groups is \emph{homotopy invariant\/} if $pr_X^*:F(X)\to F(X\times\mathbb A^1)$ is an
isomorphism for all $X\in Sm/k$.

\begin{lemma}\label{a1inv}
A presheaf with $\mC C$-correspondences belongs to $\mC S_{\mathbb A^1}$ if and only if it is homotopy invariant.
\end{lemma}

\begin{proof}
This immediately follows from the definition of $\mC S_{\mathbb A^1}$.
\end{proof}

\begin{remark}\label{BD}
Given a preadditive category $\mC B$, the category $(\mC B^{\op},\Ab)$ is naturally 
equivalent to $(\mC D^{\op},\Ab)$, where $\mC D$ is a full subcategory of any family of finitely
generated projectives of $(\mC B^{\op},\Ab)$ containing the representable functors $\mC B(-,b)$.
In other words, we can always add as many finitely generated projectives as is necessary 
\big(see, e.g.,~\cite[Section~4]{GG}\big). Recall that the Yoneda Lemma identifies $\mC B$ with the full
subcategory of representable functors of $(\mC B^{\op},\Ab)$. For example, we can add 
$\mathbb Z_{\mC C}^{\mathbb A^1}(X)$, $X\in Sm/k$, to $\mC C$ and form a bigger
preadditive category $\widetilde{\mC C}$ without changing $Pre(\mC C)$ in the above sense
that there is a natural equivalence $Pre(\widetilde{\mC C})\xrightarrow{\sim} Pre(\mC C)$. This
equivalence is induced by the restriction from $\widetilde{\mC C}$ to $\mC C$. We 
shall tacitly use this remark below when applying recollements from preceding sections to
the corresponding categories associated with $Pre(\mC C)$.
\end{remark}

It follows from Theorem~\ref{Egg2} that there is a recollement of Abelian categories
\begin{diagram*}[column sep=huge]
\mC S_{\mathbb A^1}\arrow[r,"i"] &{Pre(\mC C)}
   \bendR{i_R}
   \bendL{i_L}
   \arrow[r,"\ell"]			
& {Pre(\mC C)/\mC S_{\mathbb A^1}}\bendR{\ell_R} 
   \bendL{\ell_L} 
\end{diagram*}
with functors $i_L,i,i_R$ being from Theorem~\ref{Egg}. The functor $\ell_R$ is the inclusion and
$\ell_L:=r_L\circ\varkappa^{-1}$, where $\varkappa:Pre\bigl(\mC C^{\mathbb A^1}\bigr)\to Pre(\mC C)/\mC S_{\mathbb A^1}$
is a natural equivalence of Grothendieck categories~\cite[4.10]{GG} and $r_L:Pre\bigl(\mC C^{\mathbb A^1}\bigr)\to Pre(\mC C)$
is the left Kan extension associated with the fully faithful embedding $\mC C^{\mathbb A^1}\to\widetilde{\mC C}$ composed
with the natural equivalence $Pre(\widetilde{\mC C})\xrightarrow{\sim} Pre(\mC C)$ of Remark~\ref{BD}.

In this particular recollement we are able to give an explicit description of the functor $i_L$.
For this, recall that the \emph{strict homotopization\/} of a presheaf of Abelian groups $F$, denoted by
$[F]$, is defined as $\Coker\bigl(i_1^*-i_0^*:F(-\times\mathbb A^1)\to F\bigr)$, where $F(-\times\mathbb A^1)$
is the presheaf $X\in Sm/k\mapsto F(X\times\mathbb A^1)$ and $i_1:X\to X\times\mathbb A^1$ identifies
$X$ with $X\times 1\subset X\times\mathbb A^1$. Notice that $[F]$ is the zeroth homology presheaf
of the Suslin complex $C_*(F)$ of $F$ \big(see~\cite[Section~1]{SV1} for the definition of $C_*(F)$\big). In particular, $[F]$
is homotopy invariant, because all homology presheaves of $C_*(F)$ are homotopy invariant 
\big(see~\cite[Section~1]{SV1} or~\cite[Section~2.3]{MV} for details\big). By Lemma~\ref{a1inv} $[F]\in\mC S_{\mathbb A^1}$.

\begin{proposition}\label{iL}
$i_L:Pre(\mC C)\to\mC S_{\mathbb A^1}$ is computed as the strict homotopization functor $F\mapsto[F]$.
\end{proposition}

\begin{proof}
Our statement will follow if we show that the strict homotopization functor is left adjoint to the inclusion
$i:\mC S_{\mathbb A^1}\to Pre(\mC C)$. Let $F\in Pre(\mC C)$, $H\in\mC S_{\mathbb A^1}$ and let
$f:F\to H$ be a morphism in $Pre(\mC C)$. Consider a commutative diagram
\begin{diagram*}
F(-\times\mathbb A^1)\dar\rar{i_1^*-i_0^*}&F\dar{f}\rar{j}&{[F]}\dar{[f]}\\
                      H(-\times\mathbb A^1)\rar{i_1^*-i_0^*}&H\rar&{[H]}
\end{diagram*}   
By Lemma~\ref{a1inv} $H$ is homotopy invariant, hence the left bottom arrow equals zero and 
$H=[H]$. We see that $f=[f]\circ j$, and so the strict homotopization functor is left adjoint to $i$.
\end{proof}

Next denote by $\mathbf{PDM}^{eff}_{\mC C}(k)$ the full triangulated subcategory of $\D\bigl(Pre(\mC C)\bigr)$
consisting of the complexes with homotopy invariant homology presheaves.
Theorem~\ref{lambda} and Lemma~\ref{a1inv} imply that there exists a recollement of triangulated categories 
\begin{diagram*}[column sep=huge]
\mathbf{PDM}^{eff}_{\mC C}(k)\arrow[r,"\iota"] &{\D(Pre\bigl(\mC C)\bigr)}
   \bendR{\iota_R} 
     \arrow[r,"\lambda"]
	 \bendL{\iota_L}		
& {\D\bigl(Pre(\mC C)/\mC S_{\mathbb A^1}\bigr)}
   \bendR{\lambda_R} 
   \bendL{\lambda_L}
\end{diagram*}
The functor $\iota$ is inclusion, $\lambda$ applies the $\mC S_{\mathbb A^1}$-localization exact 
functor to any complex of $\D\bigl(Pre(\mC C)\bigr)$.
The functor $\lambda_L:=\rho_L\circ\varkappa^{-1}$, where $\rho_L$ is from Theorem~\ref{TA}
and $\varkappa:Pre\bigl(\mC C^{\mathbb A^1}\bigr)\to Pre(\mC C)/\mC S_{\mathbb A^1}$
is a natural equivalence of Grothendieck categories~\cite[4.10]{GG}. 
The functor $\lambda_R$ takes $Y$ to $\ell_R(kY)$, where 
$kY$ is a \K-injective resolution of $Y$ in $\D\bigl(Pre(\mC C)/\mC S_{\mathbb A^1}\bigr)$ and $\ell_R$
is the inclusion functor $Pre(\mC C)/\mC S_{\mathbb A^1}\to Pre(\mC C)$ which applies to $kY$ degreewise.

\begin{remark}
In fact, there is another description of $\iota_L$ and $\lambda_L$ for this particular 
recollement. Similar to~\cite[Section~2.3]{MV} one can show that $\iota_L$ is given by the Suslin complex
functor $C_*$. Here $C_*$ is a kind of homotopy extension for Proposition~\ref{iL} about the strict homotopization
functor. Since $\varkappa:Pre\bigl(\mC C^{\mathbb A^1}\bigr)\to Pre(\mC C)/\mC S_{\mathbb A^1}$ is an
equivalence of Grothendieck categories~\cite[4.10]{GG} and $\D\bigl(Pre(\mC C^{\mathbb A^1})\!\bigr)$ has enough
\K-projectives (in the sense that every complex admits a \K-projective resolution), then 
$\lambda_L$ takes a complex $X$ to the \K-projective resolution $\widetilde X$ of $\varkappa^{-1}(X)$ in 
$\D\bigl(Pre(\mC C^{\mathbb A^1})\bigr)$ and applies $\ell_L$ to $\widetilde X$ degreewise.
We use here the fact that $\ell_L$ is extended to a left Quillen functor
$\Ch\bigl(Pre(\mC C^{\mathbb A^1})\bigr)\to\Ch\bigl(Pre(\mC C)\bigr)$ with respect to the standard projective model
structure on both categories.
\end{remark}

In order to apply Theorem~\ref{bigthm} to the derived category $\D\bigl(Shv(\mC C)\bigr)$ of the Grothendieck 
category of Nisnevich sheaves 
with $\mC C$-correspondences, we now suppose that $\mC C$ is a strict 
$V$-category of correspondences in the sense of~\cite{Gar17}.
For example, the category of finite correspondences $Cor$ in the sense of Suslin--Voevodsky~\cite{SV1},
$K_0$ and $K^\oplus_0$ in the sense of Grayson--Walker~\cite{Gr,Wlk} are strict $V$-categories of correspondences
whenever the base field $k$ is perfect. The category of Milnor--Witt correspondences $\widetilde{Cor}$ 
in the sense of Calm\`es--Fasel~\cite{CF} is a strict $V$-category of
correspondences if $k$ is infinite perfect with $char(k)\not=2$. We also refer the reader to~\cite{DK} for further examples.
If $\mC C$ is a strict $V$-category of
correspondences then so is $\mC C_R:=\mC C\otimes R$ with $R$ a ring of fractions of $\mathbb Z$ like,
for example, $\mathbb Z[\frac1p]$ or $\mathbb Q$. 

Let $(\cdot)_{\mC Q}:Pre(\mC C)\to Shv(\mC C)$ be the Nisnevich sheafification functor from pre\-sheaves 
to Nisnevich sheaves with $\mC C$-correspondences and let $\mC Q=\Ker\bigl((\cdot)_{\mC Q}\bigr)$ be the 
localizing subcategory in $Pre(\mC C)$ of Nisnevich locally trivial presheaves. Similar to Lemma~\ref{nissh} we have
$Shv(\mC C)=Pre(\mC C)/\mC Q$. Since $\mC C$ is a strict 
$V$-category of correspondences, then $(\cdot)_{\mC Q}$ takes homotopy invariant presheaves to
homotopy invariant sheaves. Using Lemma~\ref{a1inv} $\mC S_{\mathbb A^1}$ satisfies
the strict Voevodsky property with respect to $\mC Q$ in the sense of Definition~\ref{vsprop}.
By Lemma~\ref{SAQ} the full subcategory of homotopy invariant $\mC C$-sheaves $\mC S_{\mathbb A^1}^{\mC Q}$
is localizing in $Shv(\mC C)$. By Lemma~\ref{L5} one has $Shv(\mC C)/\mC S_{\mathbb A^1}^{\mC Q}=
Pre(\mC C)/\mC J_{\mathbb A^1}$, where $\mC J_{\mathbb A^1}$ is the join of $\mC S_{\mathbb A^1}$ and $\mC Q$.
By Proposition~\ref{join} $Shv(\mC C)/\mC S_{\mathbb A^1}^{\mC Q}$ consists of the sheaves which are
$\mC S_{\mathbb A^1}$-local in $Pre(\mC C)$.

By~\cite[Section~6]{GP} $\D\bigl(Shv(\mC C)\bigr)$ is a compactly generated triangulated category with 
compact generators given by representable $\mC C$-sheaves $\bigl\{\mC C(-,X)_{nis}\bigr\}_{X\in Sm/k}$.
Since $\mC C$-presheaves $\bigl\{\mC C(-,X)\bigr\}_{X\in Sm/k}$ regarded as complexes in single degrees 
are compact generators of $\D\bigl(Pre(\mC C)\bigr)$, it follows that
the sheafification functor $(\cdot)_{\mC Q}:\D\bigl(Pre(\mC C)\bigr)\to\D\bigl(Shv(\mC C)\bigr)$ respects compact objects.

Following Voevodsky~\cite{Voe2}, denote by $\mathbf{DM}^{eff}_{\mC C}(k)$ the full triangulated 
subcategory of $\D\bigl(Shv(\mC C)\bigr)$ consisting of homotopy invariant cohomology sheaves. We call it
the \emph{triangulated category of $\mC C$-motives}. In the notation of
Definition~\ref{vsprop}, $\mathbf{DM}^{eff}_{\mC C}(k)=\mC E^{\mC Q}$.

If we document the above arguments and use Theorem~\ref{bigthm}, we have shown the following

\begin{theorem}\label{shvthm}
Suppose $\mC C$ is a strict  $V$-category of correspondences in the sense of~\cite{Gar17}.
There exists a recollement of triangulated categories 
\begin{diagram*}[column sep=huge]
\mathbf{DM}^{eff}_{\mC C}(k)\arrow[r,"\iota^\mC Q"] &{\D\bigl(Shv(\mC C)\bigr)}
   \bendR{\iota_R^\mC Q} 
   \bendL{C_*}
   \arrow[r,"\lambda^\mC Q"]			
& {\D\bigl(Shv(\mC C)/\mC S_{\mathbb A^1}^{\mC Q}\bigr),}
   \bendR{\lambda_R^\mC Q} 
   \bendL{\lambda_L^\mC Q}
\end{diagram*}
where $C_*$ is the Suslin complex functor,
the functor $\lambda^{\mC Q}$ is induced by the $\mC S_{\mathbb A^1}^{\mC Q}$-local\-ization functor 
$Shv(\mC C)\to Shv(\mC C)/\mC S_{\mathbb A^1}^{\mC Q}$ associated with the localizing subcategory
$\mC S_{\mathbb A^1}^{\mC Q}$ of $Shv(\mC C)$, $\iota^{\mC Q}$ is inclusion,
and $\lambda_R^\mC Q$ is induced by the \K-injective resolution functor.
\end{theorem}

Thus the Voevodsky triangulated category of $\mC C$-motives $\mathbf{DM}^{eff}_{\mC C}(k)$ fits 
into a recollement of explicit triangulated categories. The fact that $C_*$ is left adjoint to the inclusion 
is a celebrated theorem of Voevodsky~\cite[3.2.6]{Voe2} (see~\cite[1.12]{SV1} as well).


\begin{thebibliography}{99}

\bibitem{AlGG} H. Al Hwaeer, G. Garkusha, {\emph{Grothendieck categories of enriched functors}},
    J. Algebra 450 (2016), 204-241.
\bibitem{Borceux}	F. Borceux, \emph{Handbook of Categorical Algebra 2}, Cambridge University Press,  Cambridge, 1994.
\bibitem{CF} B.~Calm\`es, J. Fasel, \emph{The category of finite $MW$-correspondences}, preprint arXiv:1412.2989v2.
\bibitem{DK} A. Druzhinin, H. Kolderup, \emph{Cohomological correspondence categories}, Alg. Geom. Topology, to appear.
\bibitem{DRO} B. I. Dundas, O. R\"ondigs, P. A. {\O}stv{\ae}r,
              \emph{Enriched functors and stable homotopy theory}, Doc. Math. {8} (2003), 409-488.
\bibitem{GG} G. A. Garkusha, \emph{Grothendieck categories}, Algebra i Analiz 13(2) (2001),
            1-68 (Russian). Engl. transl. in St. Petersburg Math. J. 13(2) (2002), 149-200.
\bibitem{Gar09} G. Garkusha, {\emph{Classifying finite localizations of quasicoherent sheaves}},	
    Algebra i Analiz 21(3) (2009), 93-129. Engl. transl. in St. Petersburg Math. J. 21(3) (2010), 433-458.
\bibitem{Gar17} G. Garkusha, \emph{Reconstructing rational stable motivic homotopy theory}, Compos. Math. 155(7) (2019), 1424-1443.
\bibitem{GGME} G. Garkusha, D. Jones, {\emph{Derived categories for Grothendieck categories of enriched
    functors}}, Contemp. Math. 730 (2019), 23-45.
\bibitem{GP} G. Garkusha, I. Panin, \emph{K-motives of algebraic varieties}, Homology, Homotopy Appl. 14(2) (2012), 211-264.
\bibitem{Gr} D. Grayson, \emph{Weight filtrations via commuting automorphisms}, K-Theory 9 (1995), 139-172.
\bibitem{Krause2} H. Krause, \emph{Localization theory for triangulated categories}, London Math. Soc. Lecture Note Ser., 375, 
    Cambridge University Press, Cambridge, 2010, pp. 161-235.
\bibitem{Mac2} S. Mac Lane, l. Moerdijk, {\emph{Sheaves in Geometry and Logic: a first introduction to topos
    theory}}, Universitext, Springer-Verlag, New York, 1992.
\bibitem{Mitchell} B. Mitchell, \emph{Rings with Several Objects}, Adv. Math. 8(1) (1972), 1-161.
\bibitem{Mor} F. Morel, \emph{The stable $\mathbb A^1$-connectivity theorems}, K-theory 35 (2006), 1-68.
\bibitem{MV} F. Morel, V. Voevodsky, \emph{$\mathbb A^1$-homotopy theory of schemes}, Publ. Math. IHES 90 (1999), 45-143.
\bibitem{Nee96} A. Neeman,  {\emph{The Grothendieck duality theorem via Bousfield's techniques
and Brown representability}}, J. Amer. Math. Soc. 9(1) (1996), 205-236.
\bibitem{P} C. Psaroudakis,  {\emph{Homological theory of recollements of
abelian categories}},  J. Algebra 398 (2014), 63-110.
\bibitem{PV}C. Psaroudakis, J. Vit\'oria, {\emph{Recollements of
module categories}},  Appl. Categ. Struct. 22(4) (2014), 579-593.
\bibitem{Spal} N. Spaltenstein, \emph{Resolutions of unbounded complexes},
	Compos. Math. 65(2) (1988), 121-154.
\bibitem{Stein} B. Steinberg, {\emph{Representation Theory of Finite Monoids}}, Universitext, Springer, 2016.
\bibitem{SV1} A. Suslin, V. Voevodsky, \emph{Bloch--Kato conjecture and motivic cohomology with
         finite coefficients}, The Arithmetic and Geometry of Algebraic Cycles
         (Banff, AB, 1998), NATO Sci. Ser. C Math. Phys. Sci., Vol. 548, Kluwer Acad. Publ., Dordrecht (2000), pp. 117-189.
\bibitem{Stack} The Stacks project, https://stacks.math.columbia.edu
\bibitem{Voe1} V. Voevodsky, {\emph {Cohomological theory of presheaves with
             transfers}}, in Cycles, Transfers, and Motivic Ho\-mo\-logy Theories, Ann.~Math.~Studies, Princeton University Press, 2000, pp. 87-137.
\bibitem{Voe2} V. Voevodsky, \emph{Triangulated category of motives over a field}, in Cycles, Transfers, 
             and Motivic Ho\-mo\-logy Theories, Ann.~Math.~Studies, Princeton University Press, 2000, pp. 188-238.
\bibitem{Wlk} M. E. Walker, \emph{Motivic cohomology and the K-theory of
         automorphisms}, PhD Thesis, University of Illinois at Urbana-Champaign, 1996.

\end{thebibliography}
\end{document}